\documentclass[10pt,a4paper,reqno]{elsarticle} 

\usepackage{latexsym}
\usepackage{verbatim}
\usepackage{epsfig}
\usepackage{rotating}
\usepackage{amssymb}
\usepackage[T1]{fontenc}
\usepackage{afterpage}
\usepackage{color}
\usepackage{dsfont}
\usepackage{pstricks}

\usepackage{amssymb,amsfonts,amsmath,stmaryrd,amsthm}

\newtheorem{Theorem}{Theorem}%[section]
\newtheorem{Lemma}[Theorem]{Lemma}
\newtheorem{Proposition}[Theorem]{Proposition}
\newtheorem{Definition}[Theorem]{Definition}
\newtheorem{Corollary}[Theorem]{Corollary}

\newtheorem{Conjecture}[Theorem]{Conjecture}

\newfont{\bbold}{msbm10 scaled \magstep1}
\newfont{\bbolds}{msbm7 scaled \magstep1}

\newcommand{\ns}{\mathbb{N}}%{\mbox{\bbold N}}

\newcommand{\zs}{\mathbb{Z}}%{\mbox{\bbold Z}}

\newcommand{\qs}{\mathbb{Q}}%{\mbox{\bbold Q}}

\newcommand{\cs}{\mathbb{C}}%{\mbox{\bbold C}}
\newcommand{\fps}{formal power series}

\newcommand{\GK}{\mathbb{K}}

 \def\NW{{\sf NW}}\def\NN{{\sf N}}
 \def\SE{{\sf SE}}\def\EE{{\sf E}}
 \def\SS{{\sf S}}
 \def\WW{{\sf W}}
 \def\ES{{\sf ES}}
\def\WN{{\sf WN}}

\newcommand{\cC}{\mathcal C}
\newcommand{\cD}{\mathcal D}

\newcommand{\cV}{\mathcal V}
\newcommand{\cW}{\mathcal W}

\newcommand{\Q}{Q}
\newcommand{\Qa}{\mathcal Q}

\newcommand{\Qta}{\tilde {\mathcal Q}}

\newcommand{\C}{\mathcal C}

\newcommand{\eps}{\varepsilon}

\newcommand{\beq}{\begin{equation}}
\newcommand{\eeq}{\end{equation}}
\newcommand{\gf}{generating function}

\def\emm#1,{{\em #1}}
\newcommand{\p}{permutation}
\newcommand{\ps}{permutations}
\newcommand{\vareps}{\varepsilon}

\newcommand{\bx}{\bar x}
\newcommand{\by}{\bar y}
\newcommand{\bX}{\bar X}
\newcommand{\bY}{\bar Y}
\newcommand{\A}{\mathcal A}
\newcommand{\B}{\mathcal B}
\newcommand{\Ra}{\mathcal W}
\newcommand{\Ha}{\mathcal H}

\newcommand{\Sp}{S^{\bullet}}
\newcommand{\coSp}{s^{\bullet}}
\newcommand{\Sb}{S^{\circ}}
\newcommand{\Qc}{{Q^{\circ}}}

\newcommand{\xc}{t_c}

\begin{document}
\begin{frontmatter}

\title{Permutations sortable by two stacks in parallel and quarter plane walks}

\author{Michael Albert}
\address{Department of Computer Science,
University of Otago,
PO Box 56,
Dunedin 9054,
New Zealand}
\ead{michael.albert@cs.otago.ac.nz}
%\thanks{merci merci } 

\author{Mireille Bousquet-M\'elou}
\address{CNRS, LaBRI, Universit\'e de Bordeaux, 351 cours de la
  Lib\'eration,  33405 Talence Cedex, France} 
\ead{bousquet@labri.fr}

%*********************** Abstract ******************************
\begin{abstract}
At the end of the 1960s, Knuth characterised 
the permutations that can be sorted 
 using a stack in terms of forbidden patterns.
He also showed that they are in bijection with Dyck paths and thus counted by
the Catalan numbers. Subsequently, Even \& Itai, Pratt and Tarjan 
studied permutations that can be sorted using two
stacks in parallel. This problem  is significantly harder. In
particular, a sortable permutation can now be sorted by several
distinct sequences of stack operations. 
Moreover, 
in order to be sortable, a permutation must avoid infinitely many patterns.
The associated counting question has remained open for 40 years.
We solve it by  giving a pair of functional equations
that characterise
the generating function of permutations that can be sorted with two parallel stacks.

The first component of this
system describes the generating function $Q(a,u)$ of 
square lattice loops
 confined to the positive quadrant, counted by the length and the number of
North-West and East-South factors. Our analysis of the asymptotic
number of sortable permutations relies at the moment on two intriguing conjectures
dealing with the series $Q(a,u)$. We prove that they hold for 
loops confined to the upper half plane, or not confined at all. They remain
open for quarter plane loops. 
Given the recent activity on walks
confined to cones, we believe them to be attractive
\emph{per se}.
\end{abstract}


\end{frontmatter}

%%%%%%%%%%%%%%%%%%%%%%%%%%%%%%%%%%%%%%%%%%%%%%%%%%%%%%%%%%%%
\section{Introduction}
%%%%%%%%%%%%%%%%%%%%%%%%%%%%%%%%%%%%%%%%%%%%%%%%%%%%%%%%%%%%

If we have a device whose only ability is to rearrange certain
sequences of objects, it is  natural to ask ``What
rearrangements can my device produce?'' When the device is an abstract
one that can operate on sequences
%collections
 of any size, this becomes a
combinatorial question. Such questions were apparently first
considered by Knuth~\cite{Knuth68} who dealt with the case where the device was a
stack, i.e.~a storage mechanism operating in a last in, first out
manner (Figure~\ref{fig:one-stack}).

\begin{figure}[ht]
\begin{center}
{\includegraphics[scale=0.6]{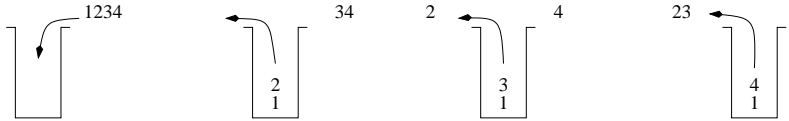}}
\caption{Four steps in the sequence of operations that outputs $2341$
  from $1234$ using a %single
 stack. Each arrow shows an  operation
  that is about to  be performed.} 
\label{fig:one-stack}
\end{center}
\end{figure}

Using a stack it is clear that the input sequence $abc$ cannot produce
the output sequence $cab$ as, in order for $c$ to be
the first element output, both $a$ and $b$ must be in the stack
together but then they will be output as $ba$ and not as $ab$. 
This is in fact the only restriction:
a permutation of an input sequence is achievable
if and only if it never
moves a later item ($c$) before two earlier items ($a$ and $b$) while
leaving the earlier items in order. 
In modern language, if the input is $12 \cdots n$, the % achievable
output permutations are those that \emm avoid the pattern, 312.
The stack operations that will
produce an output sequence from a given input sequence are easily seen
to be uniquely determined. So, it is also routine to count such
permutations and to discover that they are enumerated by the Catalan
numbers. 
This is described in Section 2.2.1 of \emph{The Art of Computer
Programming}~\cite{Knuth68}.  Knuth also establishes there similar
results for \emph{input-restricted  deques} (double-ended queues).

Knuth's investigations, nicely described in terms of ``railway yard
switching networks'', were extended by  Even \& Itai~\cite{even-itai}, Pratt~\cite{Pratt73} and Tarjan~\cite{tarjan} who
considered more general networks of stacks and queues, including the
small network consisting of two parallel stacks  that we study in
this paper (Figure~\ref{fig:two-stacks}). 
 This  work was foundational for the study of \emph{permutation
classes} which can loosely be described as those collections of
permutations 
that are closed by taking sub-permutations\footnote{To be clear, given
  a permutation of $\{1,2,\ldots,n\}$ written as a word in one line
  notation, we form a sub-permutation by taking any subword (of length
  $k$ say) and then replacing the symbols of the subword by
  $\{1,2,\ldots,k\}$ while maintaining their relative values. For
  instance taking the subword of $15324$ occurring in the second,
  fourth and fifth positions ($524$) illustrates that $312$ is a
  sub-permutation of $15324$.}. In our case, we 
observe indeed that any sub-permutation of a permutation that can be
produced using  two parallel stacks can itself be produced by this device 
 simply by ignoring any
 operations that affect elements not in the sub-permutation.
The study of permutation classes has been an active and growing field, often concentrating
on enumeration, but also dealing with structural properties of these
classes. For some general discussions and background
see the
books~\cite{bona2012combinatorics,kitaev2011patterns,MR2722945},
and~\cite{bona-discipline} for a survey on stack-sorting.

\begin{figure}[ht]
\begin{center}
{\includegraphics[scale=0.8]{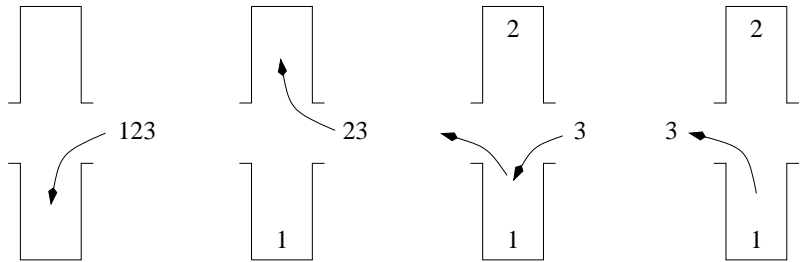}}
\caption{The permutation $312$ cannot be produced with a single stack, but can
  be produced with two parallel stacks as shown here. Note that
  several distinct  sequences of operations produce it.}  
\label{fig:two-stacks}
\end{center}
\end{figure}

Despite this activity,  most  problems related to the
rearranging power of  Knuth's 
switchyard networks have turned out to be very
hard. For networks consisting of two stacks, the case of parallel
stacks seems   a bit more manageable than that of two
stacks in series. For instance,  the list of
minimal permutations that cannot be 
produced by two parallel stacks has been known
since 1973~\cite{Pratt73}, but for two stacks in series 
it is only known to be infinite~\cite{Murphy02}. 
Similarly,  it has just been proved this year that one can decide in
polynomial time if a permutation can be sorted
 by two stacks in series~\cite{pierrot-pol},
while the corresponding result follows
from a 1971 paper for two parallel stacks~\cite{even-itai}
% mbm added
(see also~\cite{rosenstiehl}). 
However, the  questions ``How many 
 permutations of length
$n$  can be produced  by two stacks in series, or by two stacks in
parallel?'' have  remained equally open for forty years. 

 We  answer the latter question in this paper,
by giving a system of two functional equations that defines the
 generating function $\sum_n s_n t^n$, where $s_n$ is the number of permutations of length
 $n$ that can be produced with two parallel stacks.  Denton~\cite{denton}
has presented an algorithm for this problem whose complexity is $O(n^5 2^n)$
(for enumerating the sortable permutations of length $n$). The form of the
functional equations we obtain is such that we have, in principle, a polynomial time
algorithm though we have not tried to estimate its precise complexity.

 We also determine the exponential 
growth of the 
numbers $s_n$,  modulo some conjectures that deal with square lattice
walks confined to the quarter plane. These walks naturally encode the
admissible sequences of stack operations, in the same way as Dyck
paths do in the case of a single stack.
Our conjectures deal with the enumeration of quarter plane walks
counted by the
length and by the number of \emm corners, of certain types. Walks
confined to a quadrant have attracted a lot of attention in
the past decade (see
e.g.~\cite{bostan-kauers,bostan-raschel-salvy-excursions,mbm-kreweras,mbm-mishna,kurkova-raschel-nature,zeilberger-gessel}),
and we think that  our conjectures are %might be 
 interesting quite independently of the original stack sorting question.

Finally we remark that in this metaphor of ``devices rearranging
input'' there are two common viewpoints. As described above, Knuth
  tended to view the input as arriving in fixed order $12\cdots n$ and 
then the question is ``How many permutations can be
produced?''. Tarjan on
the other hand tended to think of the
objective being to sort the input permutation, so the enumerative
question becomes ``How many permutations can be sorted?''. Of course,
passing to inverses, the two viewpoints are 
equivalent to one another: if a sequence of operations produces $\pi$
from the identity,
then the same sequence, applied to $\pi^{-1}$, produces the identity. We will be adopting the first viewpoint.

% \medskip
The outline of the paper is as follows. In Section~\ref{sec:canonical}, we
describe a set of \emm canonical operation sequences, such that
each permutation that can be produced using two parallel stacks is obtained
by exactly one canonical operation sequence. In Section~\ref{sec:enumeration}, we
establish a system of functional equations that characterises the
generating function of canonical sequences, and thus, of permutations that can be produced
%(or sorted)
 by two parallel stacks. The first equation in this system
defines the generating function of quarter plane walks, weighted by their length and
the number of North-West and East-South factors (also called
\emm corners,). In Section~\ref{sec:conj}, we state two  conjectures about this
\gf, and provide evidence for them by proving that they hold if we
only impose on walks a half plane restriction, or no restriction at
all. In Section~\ref{sec:asympt} we derive from our system of equations the
exponential growth of the number of permutations of length $n$ produced by two
parallel stacks, assuming the conjectures of Section~\ref{sec:conj}. We
conclude with a few comments on our results and conjectures in
Section~\ref{sec:final}.

%%%%%%%%%%%%%%%%%%%%%%%%%%%%%%%%%%%%%%%%%%%%%%%%%%%%%%%%%%%%
\section{Canonical operation sequences}
\label{sec:canonical}
%%%%%%%%%%%%%%%%%%%%%%%%%%%%%%%%%%%%%%%%%%%%%%%%%%%%%%%%%%%%
Throughout this paper we consider the action of two
stacks in parallel, and attempt to count  permutations of length $n$
that such a machine can 
produce. These permutations are said to be \emm achievable,. 
The primary issue in this question,
as opposed to the case of  a single stack, 
%as  considered by Knuth~\cite{Knuth68}, 
is that there is no one-to-one correspondence between
 sequences of operations
of the machine and achievable  permutations.  That is, 
several sequences  of operations may produce the same
permutation:
we then say that they are \emm equivalent,.
The most  obvious case  is that of the identity permutation 
of length $n$:  there are at least $2^n$ ways to produce it using two
stacks (alternate input and output operations,  allowing the
freedom of choice as to which stack to use --- in fact there are more
ways, since we can delay some
output steps if we choose the following input to be to the other
stack).

In this section we define a family of operation sequences, called \emm
canonical,, such that each operation sequence is equivalent to exactly
one canonical sequence. Canonical sequences  are thus in one-to-one
correspondence with achievable  \ps.

In order to proceed
further, we  present three equivalent descriptions of 
operation sequences.
Recall what the basic scenario is: input items numbered consecutively from
$1$ through $n$ are processed by two stacks, each of which is capable
of containing an arbitrarily large amount of data, but whose
operations are limited to input ($I$) and output ($O$); an
output operation produces the most recently entered item (i.e.~items
are processed in a last-in first-out fashion). Items are output as a
sequence, and after all the input has been processed and the stacks
emptied, the result is a permutation of the original input (Figure~\ref{fig:two-stacks}).

Operation sequences are  encoded as words over the alphabet
$\{I_1, I_2, O_1, O_2\}$, the subscripts determining which stack is
referred to. 
Note that both stacks must be empty at the end, and that one cannot output from an empty stack. This means that a word over
$\{I_1, I_2, O_1, O_2\}$  is an 
operation sequence if and only if 
it contains the same
number of $I_i$ as $O_i$ letters for $i = 1, 2$, and, in each
prefix, the number of $I_i$ letters is at least as great as the number
of $O_i$ letters for $i = 1, 2$. 
Equivalently, it is a \emph{shuffle} of two Dyck words, one on the letters $I_1$ and $O_1$, and 
the other on the letters $I_2$ and $O_2$.
The 
\emm type, of an operation sequence is the word on $\{I, O\}$
obtained  by deleting the subscripts on its letters.

We consider   square lattice walks which begin at
$(0,0)$ and use steps  $\EE = (1,0)$, $\NN = (0,1)$, $\WW =
(-1,0)$  and $\SS = (0,-1)$. Such a walk  is a  \emm loop, if it ends at
$(0,0)$. It is a \emm quarter plane walk, if it remains in the
quadrant $\{(x, y): x\geq 0, y \ge 0\}$.
There is an obvious one-to-one correspondence between
operation sequences and  quarter plane loops (replace
$I_1$ by $\EE$, $I_2$ by $\NN$, $O_1$ by $\WW$ and $O_2$ by $\SS$). Under this
correspondence, the $(x,y)$ coordinate reached after processing a
prefix of an %legitimate
 operation sequence simply records the number of
items in each stack at that point. The number of quarter plane
loops 
consisting of $2n$ steps is well known to be
$
C_n C_{n+1},
%\frac  1{ (2n+1)(2n+4) } {2n+2\choose n+1} ^2.
$
where $C_n= {2n \choose n}/(n+1) $is the $n^{\mbox{\scriptsize th}}$ Catalan number~\cite{guy-bijections,bernardi-tree-rooted}.
Observe that the type of an operation sequence corresponds to the
projection of the associated loop on the diagonal $x=y$.

\begin{figure}[ht]
\begin{center}
{\includegraphics[scale=0.8]{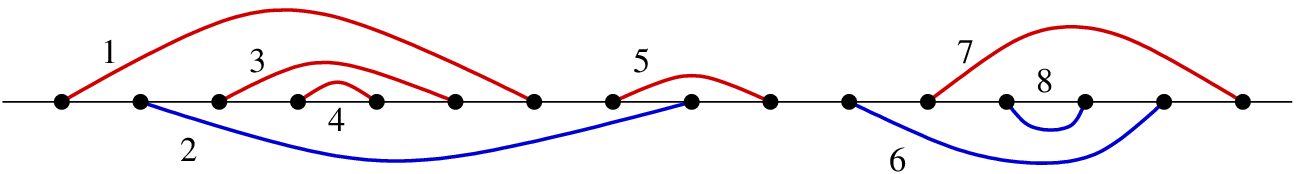}}

\vskip 8mm
{\includegraphics[scale=0.8]{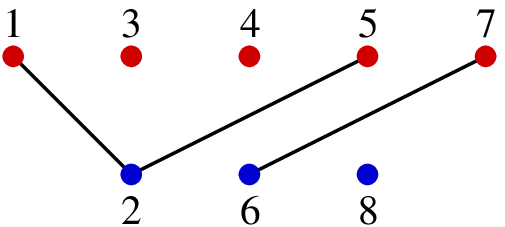}}
\caption{An illustration of the arch system associated with the operation sequence
  $I_1I_2I_1I_1O_1O_1O_1I_1O_2O_1I_2I_1I_2O_2O_2O_1$, and its
  associated graph. The arches are labelled using the left-to-right
  order of their left endpoint. This arch system has five connected
  components, and one left-right pair (between arches 2 and 5). 
%The  arches 3 and 4 are parallel. 
The output permutation is $43125867$. 
%The associated partition is $\{  \{1 ,7 \},    \{ 2,9 \},    \{ 3,6
%\},    \{4 ,6 \},    \{8 ,10 \},    \{ 11,15   \},    \{ 12,16 \},
%\{ 13,14 \}   \}$.}
} 
\label{fig:ex}
\end{center}
\end{figure}

A third perspective on these objects arises from considering them as
bi-coloured \emph{arch systems} (Figure~\ref{fig:ex}). This is the
 two-dimensional counterpart of the standard bijection between
Dyck paths and (one-coloured) arch systems~\cite[Exercise 6.19o]{stanley-vol2}. For an operation sequence of length $2n$, take 
$2n$ points arranged along a line, labelled from $1$ to $2n$. These
points represent time, that is, the $2n$ steps of the operation
sequence. For each item $k$ in $\{1, \ldots, n \}$, build an arch
joining $i$ to $j$ where $i$ (resp.~$j$) is the time at which $k$ is
input to  (resp.~output from) a stack. 
If $k$ is processed by the first
stack, the arch will be
above the line (and will be thought of as
\emph{red}), and otherwise  below the line (and thought of as
\emph{blue}).  Observe that 
the arches above the line do not cross, nor do the ones below the line
--- but there are no further
restrictions on such systems. The operation sequence is easily
recovered by scanning from left to right the $2n$ points of the arch
system, writing $I$ (resp.~$O$) if an arch opens (resp.~closes) at
this point, and $1$ (resp.~$2$) if this arch is above (resp.~below)
the line. Upon closing the supporting line into a cycle, an arch
system can also be seen  as a \emm rooted planar cubic map with a
distinguished Hamiltonian cycle,.
In this disguise, they were already considered by  Tutte~\cite{Tutte62}.

We use the following  simple terminology: 
\begin{itemize}
\item the \emph{first} arch is the one which has least left
endpoint; more generally, the $k^{\mbox{\scriptsize th}}$ arch 
 is the one with $k^{\mbox{\scriptsize th}}$ smallest left endpoint; 
\item an arch joining $i$ to $j$ \emph{moves} the element  $k$ of
  $\{1, \ldots, n\}$ that is input to a stack at time $i$ and output from at
  time $j$.
\end{itemize}
Observe that  the $k^{\mbox{\scriptsize th}}$ arch moves item $k$.

%\medskip
Our aim in this section is to describe a set of operation sequences in
bijection with achievable \ps. A first observation is that two
sequences obtained from one another by commuting pairs of adjacent
letters $I_1O_2$ or $I_2 O_1$ are equivalent. 
An operation sequence \emph{outputs eagerly} if it contains neither
$I_1 O_2$ nor $I_2 O_1$ as a factor. In other words, if the next item
of the permutation which it is producing is already present in one of
the two stacks (necessarily at the top of the stack), then it is
output immediately, before any other input  
(necessarily to the other stack) is carried out. Such sequences
correspond to walks in the plane containing no $\EE\SS$ or  $\NN\WW$  factor and
to arch systems in which the left endpoint of an arch of one colour is
never followed immediately by the right endpoint of an arch of the
opposite colour --- a configuration that we call a 
  \emph{left-right pair} (see Figure~\ref{fig:ex}). 

The following lemma is due to Pratt~\cite{Pratt73} who stated it in a
somewhat more general context and with different terminology.

\begin{Lemma}\label{lem:eager} If a permutation can be produced by some
 operation sequence, then it can be produced by one that
outputs eagerly.
\end{Lemma}

\begin{proof} 
Assign the ordering  $O_1 < O_2 < I_1 < I_2$ to the operation
letters. If an operation sequence 
$v = s I_1 O_2 t$ (respectively $s I_2 O_1 t$)   produces a permutation
$\pi$, then  $v' = s O_2 I_1 t$ (respectively $s O_1 I_2 t$) is also
an operation sequence and produces $\pi$. The
 sequence $v'$ is in each case lexicographically smaller than $v$ so
after a finite number of transformations of this type, an operation sequence generating $\pi$ is obtained that contains none of the forbidden factors.
\end{proof}

More simply we could just say that
``it can't hurt to output an element as 
soon as it is possible to do so'', which is essentially the content of
Pratt's observation.

A second source of ambiguity in operation sequences is the possibility
of reflecting  one or several (well chosen) arches in the horizontal
line. For instance, reflecting all arches 
%does not change the output \p.
gives an equivalent arch system. The same holds if we reflect one arch joining two consecutive points of
the line. Which groups of arches can one
thus reflect?

We say  that two arches of different colours  \emm cross, if they cross once
the one below the line is reflected. We sometimes
consider the arches as vertices of a graph, two arches being adjacent if they
cross (Figure~\ref{fig:ex}, bottom). This graph  is then bipartite.
We refer to its connected components
 as the \emph{(connected) 
  components} of the arch system, and call a non-empty arch system
\emph{connected} if its corresponding graph is.  In terms of operation
sequences, or equivalently quarter plane loops, this means   that no
proper factor is an operation sequence (this may be already clear to
the reader, but will be proved when enumerating connected arch
systems in Section~\ref{sec:enumeration}).  Connected components were also considered by Tutte in a
planar map context~\cite[Sec.~8]{Tutte62}.

\begin{Definition}
  An arch system or its corresponding operation sequence is \emph{standard} if  the
  first arch of each component is red (that is, above the line). It is
   \emph{canonical} if, in addition, it outputs eagerly.
\end{Definition}

\begin{figure}[ht]
\begin{center}
{\includegraphics[scale=0.8]{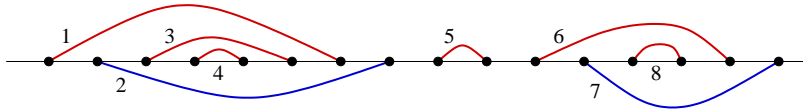}}
\caption{The canonical arch system that is equivalent to the arch system of
  Figure~\ref{fig:ex}. Note that the left-right pair created by edges
  2 and 5 in Figure~\ref{fig:ex} has disappeared (these edges
   do not cross any more). Also, the colours
  of the two rightmost components (edges 6, 7, 8) have % been
  changed.  The  output permutation is still $43125867$.}
\label{fig:standard}
\end{center}
\end{figure}

The following lemma is illustrated by Figure~\ref{fig:standard}.
\begin{Lemma}\label{lem:canonical-existence} If a permutation $\pi$
 % can be produced by some  operation sequence, 
is achievable, then it can be produced
  by a canonical  operation sequence.
\end{Lemma}
\begin{proof} By Lemma~\ref{lem:eager}, $\pi$ can be produced by a
  sequence that outputs eagerly. Let us take such a sequence $v$, and
  reflect the components that do not begin with a red arch. By
  definition of components, this does not create crossings between
  arches lying on the same side of the line, so that one obtains  another
  operation sequence $w$.  This sequence outputs eagerly since $v$ does.

It remains to prove that $w$ produces $\pi$.  But this is clear,
because the $k^{\mbox{\scriptsize th}}$ arch of $w$ moves item $k$ in and out exactly at the
  same time as the $k^{\mbox{\scriptsize th}}$ arch of $v$ does. In particular, items are
  output in the same order.  
\end{proof}

Let us now address the uniqueness of a canonical operation sequence for
each achievable \p.

\begin{Lemma} \label{lem:type}
If $v=v_1 \cdots v_{2n}$ and $w=w_1 \cdots w_{2n}$ are two %legitimate
equivalent operation sequences,
both of which output eagerly, then they have the same type. 
\end{Lemma}

\begin{proof} 
Suppose that the letters $v_i$ and $w_i$ are of the same type, for
$1\le i \le j$, and let us prove that this is also true for $v_{j+1}$
and $w_{j+1}$. After the $j^{\mbox{\scriptsize th}}$ operation,  $v$ and $w$ have performed the same
number of input and output operations, and since they are equivalent,
the items that are 
currently in the stacks according to the $v$ sequence, are the same as
those currently in the stacks according to the $w$ sequence (though
their disposition between the stacks may differ).  The items that
have not been moved yet are also the same for both sequences. Since $v$ and $w$
output eagerly, if the next item to be output is already in the stacks
(for $v$ and $w$) it will be
output immediately by both operation sequences. If not,
both must perform an input operation at this point. In
either case, the  types of the next operation in $v$ and $w$
agree.
\end{proof}

\begin{Lemma} \label{lem:affect} If $v$ and $w$ are two %legitimate
equivalent operation
sequences having the same type,
then for each $i$, the $i^{\mbox{\scriptsize th}}$ operation in $v$
moves the same item as the $i^{\mbox{\scriptsize th}}$  operation in $w$.
\end{Lemma}
\begin{proof}
Let $2n$ be the length of $v$ and $w$. Recall that the input of the
stack is the identity permutation $1 2 \cdots n$, and let us denote by $\pi_1
\cdots \pi_n$ the  permutation produced by $v$ (and $w$).  If the
$i^{\mbox{\scriptsize th}}$
% mbm Below, I put the small "th" in exponent (the ith...) several times.
operation has type $I$, and $k$ inputs have taken place before, then the
item moved by the $i^{\mbox{\scriptsize th}}$ operation is $k+1$. Similarly, if the
$i^{\mbox{\scriptsize th}}$ operation has type $O$, and $k$ outputs have taken place before, then the
item moved by the $i^{\mbox{\scriptsize th}}$ operation is
$\pi_{k+1}$. Hence $v$ and $w$ move the same item at each time.
\end{proof}

We can now conclude the discussion of this section.
\begin{Proposition}
\label{PROP:StacksAndArches}
Every achievable permutation $\pi$ is produced by a unique  canonical operation sequence.
\end{Proposition}
\begin{proof}
The existence of a canonical sequence producing $\pi$ is guaranteed by
Lemma~\ref{lem:canonical-existence}. 
Now suppose that two canonical 
 operation sequences $v$ and $w$ produce $\pi$. By
 Lemma~\ref{lem:type}, they have the same type, and by 
Lemma~\ref{lem:affect}, they move the same
element at time $i$, for each $i$. These two properties mean  that 
$v$ and $w$ only differ by the colouring of
some arches. However, once we 
colour the first arch in a component, the
colours of all 
the other arches of that component are fixed (because two 
arches that cross must have different colours). 
But $v$ and $w$ are standard, so that the first arch of each component
is red in $v$ and $w$. This implies that $v$ and $w$ coincide.
\end{proof}

It will be useful to define \emm primitive, objects. First, note
that the concatenation of two arch systems (or two operation
sequences)  $w_1$ and $w_2$ is an arch system $w$. Moreover, $w$ is
canonical if and only if $w_1$ and $w_2$ are canonical. We say that a
non-empty arch system (or 
operation sequence) is \emm primitive, if it cannot be written as a
non-trivial concatenation. This means that the corresponding quarter
plane walk only visits the origin of the lattice at the
beginning and at the end. Clearly, a 
connected arch system is primitive. An arch system
is  an  arbitrary sequence of primitive arch systems, and a similar
statement holds for canonical arch systems. The permutations produced by a primitive canonical arch system
are also said to be primitive.

% \medskip
% \noindent
{\bf Connection with results of Even \& Itai~\cite{even-itai}.} In 1971,  Even and
Itai gave the following  characterization
of permutations achievable with two parallel stacks. To a permutation $\pi=\pi_1 \cdots \pi_n$,
 associate a 
graph $G(\pi)$ with vertices $1,2, \ldots, n$ and an edge from $i$ to $j$
(with $i<j$) if there exists $k>j$ such that $kij$ is a subsequence of
$\pi$. Then $\pi$ is achievable if and only if this graph is
bicolourable. Moreover, Even \& Itai proved that in this case,  one
can produce $\pi$ by putting items out as soon as 
possible (eager output) and otherwise putting the first available item
from the input into the stack corresponding to its colour. This is
related to our results as follows: if one colours $G(\pi)$ in such a
way that the smallest element in 
each connected component is red, then the operation sequence
described by Even and Itai is exactly the canonical operation
sequence associated with $\pi$. Moreover, the graph associated with
this operation sequence (as in Figure~\ref{fig:ex}) coincides with
$G(\pi)$. Since an edge 
% from $i$ to $j$ in 
of $G(\pi)$ gives rise to a pair of crossing arches in \emm
any, operation sequence that produces $\pi$, this means  that canonical operation sequences 
%are those that 
minimise the number of arch crossings. 

%%%%%%%%%%%%%%%%%%%%%%%%%%%%%%%%%%%%%%%%%%%%%%%%%%%%%%%%%%%%%%%%%%%%
\section{Exact enumeration}
\label{sec:enumeration}
%%%%%%%%%%%%%%%%%%%%%%%%%%%%%%%%%%%%%%%%%%%%%%%%%%%%%%%%%%%%%%%%%%%%
In this section, we derive a system of functional equations that
characterises the length generating function $S(t)$ of achievable permutations 
by two stacks in parallel:
\[
S(t)= 1+t+2t^2+6t^3+23t^4+103t^5+513t^6+2760t^7+15741t^8+O(t^9).
\]
The first equation in this system characterises the generating function
$\Qa(a,u;x,y)$  of quarter plane walks,  when counted by the
length (variable $u$), the number of $\NW$ or $\ES$ corners (variable $a$), and the
coordinates of their endpoint (variables $x$ and $y$):
\begin{align*}
\Qa(a,u;x,y) = &1+(x+y)u + (2+2xy+x^2+y^2)u^2 \\ &+ \left( (a+4)(x+y) +
3x^2y +3xy^2+x^3+y^3\right)u^3 + O(u^4).
\end{align*}
By setting $x=y=0$, and replacing $u$ by $\sqrt u$, one obtains 
the generating function $Q(a,u)$ of quarter plane \emph{loops}, counted by
half-length ($u$) and $\NW$ or $\ES$ corners ($a$):
\[
Q(a,u) =1+2u+(8+2a)u^2+(44+24a+2a^2)u^3+O(u^4).
\]
Equivalently, $Q(a,u)$ counts arch systems by the number of 
arches $(u$) and the number of left-right pairs ($a$). The last
series involved in our
system  is the generating function of  standard connected arch systems, counted by the
number of arches ($v$) and the number of left-right pairs~($b$):
\[
C(b,v)= v+bv^2+ b(b+2)v^3+O(v^4).
\]
The reason why we have three different length  variables ($t$, $u$ and
$v$) and two different corner variables ($a$ and $b$) will be made clear below.

For a ring $\GK$, we denote by
$\GK[u]$ (resp.~$\GK[[u]]$) the ring of polynomials (resp.~formal power series) in
$u$ with coefficients in $\GK$. This notation is generalised to several
variables. For instance, 
$\Qa(a,u;x,y) \in \ns[a,x,y][[u]]$. 

\begin{Theorem}\label{thm:eqs}
The generating function  $\Qa(a,u;x,y)\equiv \Qa(x,y)$  of quarter plane walks is
characterised by the following equation: 
\begin{multline} \label{eq:Qa}
  (1-u(x+\bx+y+\by)-u^2(a-1)(x\by+y\bx))\Qa(x,y)=\\
1-u\by (1+ux(a-1))\Qa(x,0)-u\bx (1+uy(a-1))\Qa(0,y) ,
\end{multline}
where $\bx=1/x$ and $\by=1/y$.
The generating function for quarter plane loops is %thus 
\[
\Q(a,u)=\Qa(a,\sqrt u; 0,0).
\]
  The generating function $C(b,v)$ for connected standard arch systems is
characterised by
 \begin{equation}
\label{FE:QC}
\Q(a,u) 
%= 1 + 2C\left( \frac{\Q-1+a}{\Q} ,u \Q^2 \right)
 = 
1 + 2C\left( 1 - \frac{1-a}{\Q} , \, u\Q^2\right),
\end{equation}
where $Q$ stands for $Q(a,u)$.
Finally, 
the generating function $S(t)\equiv S$ of permutations that can be produced by
two parallel stacks is characterised by
\begin{equation}
\label{FE:SC}
S(t) = 1 + C \left( 1 - \frac{1}{S}, \, tS^2 \right).
\end{equation}
\end{Theorem}
\begin{proof}
   The equation defining $\Qa(x,y)$ translates  a simple recursive
   description of quarter plane walks, according to which a walk is:
   \begin{itemize}
   \item either empty,
\item or obtained by adding an \EE\  (resp.~\NN) step at the end of
  another quarter plane   walk,
\item or obtained by adding an \ES\  (resp.~\NW) corner to a walk that does
  not end on the $x$- (resp.~$y$-) axis,
\item or obtained by adding a \SS\  (resp.~\WW) step to a walk that does
  not end on the $x$- (resp.~$y$-) axis and whose final step is not \EE\ 
  (resp.~\NN).
   \end{itemize}
Moreover, these four  cases are disjoint. We  now write the
contribution to $\Qa(x,y)$ of each case, using the following basic
remarks:
\begin{enumerate}
\item[--] the  generating function of walks ending with an \EE\ (resp.~\NN) step is
  $ux \Qa(x,y)$  (resp.~$uy \Qa(x,y)$),
%\item[--] the  generating function of walks ending with an step is ,
\item[--] the generating function of walks ending on the $x$- (resp.~$y$-) axis is
  $\Qa(x,0)$ (resp.~$\Qa(0,y)$).
%\item[--] the generating function of walks ending on the $y$-axis is
\end{enumerate}
 These two observations allow us to express
$\Qa(x,y)$ as follows:
 \begin{eqnarray*}
\Qa(x,y)&=& 1 + u (x+y) \Qa(x,y)\\
&+ &a u^2x \by
  \left(\Qa(x,y)-\Qa(x,0)\right)+ a u^2\bx y \left(\Qa(x,y)-\Qa(0,y)\right)  \\
&+&u\by \left( \Qa(x,y)-\Qa(x,0)-ux\Qa(x,y)+ux\Qa(x,0)  \right)\\
&& \hskip 20mm +u\bx \left( \Qa(x,y)-\Qa(0,y)-uy\Qa(x,y)+uy\Qa(0,y)  \right).
 \end{eqnarray*}
This gives the first equation of the proposition. It is equivalent to
a recurrence relation defining the coefficient of $u^n$ in $\Q(a,u)$, and
thus characterises this series.

\medskip
Let us now relate the series $Q(a,u)=\Qa(a,\sqrt u; 0,0)$ and $C(b,v)$.
Let $w$ be a non-empty quarter plane loop, or equivalently an arch
system. The first arch of $w$ belongs to some connected component $c$, which
may be standard or not. The arches of $w$ that do not belong to $c$ do
not cross the edges of $c$.   So the whole system $w$ is
obtained by inserting an arch system between each pair of
adjacent points of $c$, and after the last point of $c$
(Figure~\ref{fig:insertion}). If $c$ has $n$ arches then there are $2n$
positions to make such insertions. Ignoring the %second
corner parameter for
the moment we obtain:  
\[
\Q(1,u) = 1 + 2C(1, uQ^2),
\] 
where $Q$ stands for $Q(1,u)$. On the right-hand side, the factor $2$
corresponds to the choice of colour 
for the first arch (since the series $C$ only counts \emm standard,
connected arch systems), the $u$ enumerates the arches of $c$, and the 
$Q^2$ allows for  the inserted arch systems. It remains
to account for %the second variable, that is, for 
the number of
left-right pairs. If an arch system $w_1$ is inserted between 
% a pair of arches 
the endpoints of two arches of $c$ that do not form a left-right pair
in $c$, then the only 
left-right pairs it creates are those that are already present in
$w_1$. If $w_1$ is non-empty and inserted in a left-right pair of $c$,
then  it destroys that left-right pair,
but adds any that it might contain itself. Hence, a connected arch
system $c$ with $n$ arches and $k$ left-right pairs contributes $v^n b^k$
in $C(b,v)$ and gives rise, by insertion of $2n$ arch systems, to a set of
arch systems counted by
$$
 u^n Q(a,u)^{2n-k} \left( a+ \left(Q(a,u)-1\right)\right)^k.
$$
The term $Q(a,u)^{2n-k}$ corresponds to insertions in places that are
not left-right pairs, while each left-right pair gives rise to a term
$a$ (insertion of an empty system $w_1$) and a term $Q(a,u)-1$
(insertion of a non-empty $w_1$). This gives~\eqref{FE:QC} by summing
over all possible values of $n$ and $k$ and multiplying by 2 (since
$c$ is not necessarily standard).

In order to prove that this equation  uniquely defines
$C(b,v)=\sum_{k,n}c_{k,n}b^k v^n$, it
suffices to extract from~\eqref{FE:QC} the coefficient of $a^k u^n$:
this gives an expression of $c_{k,n}$ in terms of the coefficients $c_{\ell,
  m}$ for $m<n$ and of the coefficients of $Q$.

\begin{figure}[ht]
\begin{center}
{\includegraphics[scale=0.7]{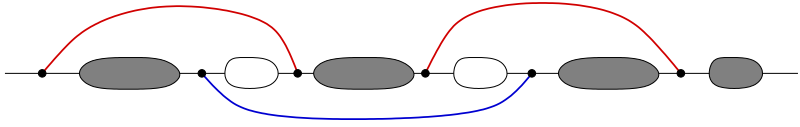}}
\caption{The structure of an arch system: a connected system $c$ with $n$
  arches (here, $n=3$), in which $2n$ arbitrary arch systems are inserted. Here, $c$
  has two left-right pairs. The arch systems that are inserted there
  (shown in white) destroy these left-right pairs, unless they are empty.} 
\label{fig:insertion}
\end{center}
\end{figure}

\medskip

Using the same argument we can finally derive a functional equation for
the generating function $S(t)$ of achievable permutations. Indeed, Proposition
\ref{PROP:StacksAndArches} tells us that they are in bijection with
canonical arch systems, that is, with standard arch systems having no
left-right pairs. Such systems $w$ are 
obtained  from a  (standard) connected system $c$ as  before, but all
inserted arch systems must be canonical;
moreover, one cannot insert an empty system in a left-right pair of
$c$. Hence a connected arch
system $c$ with $n$ arches and $k$ left-right pairs  gives rise to a set of
canonical arch systems  counted by
$$
 t^n S(t)^{2n-k} \left(S(t)-1\right)^k.
$$
This gives~\eqref{FE:SC} by summing over all possible values of $n$
and $k$. This equation is equivalent to a recurrence relation defining
the coefficient of $t^n$ in $S(t)$, and thus characterises the series $S(t)$.
\end{proof}

It will be convenient to relate the functional equation~\eqref{FE:QC}
to a compositional inversion in the ring $\qs[[a,t]]$ of bivariate
power series with rational coefficients. 

\begin{Proposition}
  Let $\Q(a,u)$ and $C(b,v)$ be defined as above,
  and define the bivariate series $ A, U, B $ and $V$ as follows:
\beq\label{UBAT}
\left\{ 
\begin{array}{lllllllllll}
A(b,v) &=&\displaystyle 1 + (1+2C(b,v))(b-1), 
\\
\\
U(b,v) &=& \displaystyle \frac{v}{(1+2C(b,v))^2}, 
\end{array}
\right. \hskip 8mm
\left\{ 
\begin{array}{lllllllllll}
B(a,u) &=& \displaystyle  1 - \frac{1-a}{Q(a,u)}. 
\\
\\
V(a,u) &=& \displaystyle u Q(a,u)^2.
\end{array}
\right.
\eeq
Then it follows from~\eqref{FE:QC} that 
$$
A(B(a,u),V(a,u))=a \quad \hbox{and} \quad U(B(a,u),V(a,u))=u,
$$
so that, by inversion in $\qs[[a,u]]$,
\beq\label{BAT}
B( A(b,v),U(b,v))=b \quad \hbox{and} \quad V(A(b,v),U(b,v))=v.
\eeq
 The identity  $B(A(b,v), U(b,v))=b $ can be rewritten as
\beq\label{CfromQ}
%Q(a,t) &=& 1 + 2C(B,U) \\
Q(A(b,v),U(b,v)) = 1 + 2C(b,v).
\eeq
\end{Proposition}
The proof is an elementary calculation.
A consequence is that we can eliminate the series $C(b,v)$ from the system of
Theorem~{\rm\ref{thm:eqs}}, and thus obtain an equation defining $S(t)$ in terms
of $\Q(a,u)$. This relation looks nicer when we introduce the
generating function $\Sp(t)$ that counts \emm primitive, canonical operation sequences,
defined at the end of Section~\ref{sec:canonical}. 

\begin{Corollary}\label{cor:QS}
  The series  $\Q(a,u)$, $S(t)\equiv S$ and $\Sp(t)\equiv \Sp$ that
  count  quarter plane 
  loops, achievable permutations  and primitive
 achievable permutations respectively, are related by
 $$
S=\frac 1 {1-\Sp}
$$
and
$$
     \Q\left(-\Sp ,\frac t {(1+\Sp )^2}\right)= \frac{1+\Sp }{1-\Sp }.
$$
The second equation characterises $\Sp$ in terms of $Q$.
\end{Corollary}
\begin{proof}
 The first identity is a direct consequence of the definition of
 primitive  achievable  \ps. For the second one,
 specialise~\eqref{CfromQ} to $b=1-1/S$ and  $v=tS^2$,  
and use~\eqref{FE:SC}.
\end{proof}
The equation above is the most efficient way we have found to compute
the coefficients of $\Sp$ and $S$.

%%%%%%%%%%%%%%%%%%%%%%%%%%%%%%%%%%%%%%%%%%%%%%%%%%%%%%%%%%%%%%%%%%%%
\section{Corners in square lattice walks}
\label{sec:conj}
%%%%%%%%%%%%%%%%%%%%%%%%%%%%%%%%%%%%%%%%%%%%%%%%%%%%%%%%%%%%%%%%%%%%

Our analysis of the asymptotic behaviour of the number of %sortable
achievable permutations of length $n$, performed in the next section,  relies on three
conjectures which have intrinsic combinatorial interest.   

\begin{Conjecture}\label{conj:a+1}
The series $\Q(a,u)$ is $(a+1)${\rm -positive}. That is, it can be expanded as
\[
\Q(a,u)= \sum_{n\ge 0} u^n P_n(a+1),
%\left( \sum_{k=0}^n q_{k,n}\, (a+1)^k \right)  ,
\]
where %$q_{k,n} \geq 0$.
$P_n(x) \in \ns[x]$.
\end{Conjecture}
Of course, it is combinatorially clear that $\Q(a,u)$ is a power
series in $u$ with coefficients in $\ns[a]$, and hence in
$\zs[a+1]$. What is not clear is why the coefficient of $(a+1)^k$
should be non-negative. This has been checked on a computer up to
half-length $n=100$, using the functional equation~\eqref{eq:Qa}.

Much of our analysis depends on being able to estimate the radius of convergence of various bivariate series as a function of one of the variables. In this context the name of the other variable is not important (and indeed we will most often be subsituting more or less complicated expressions for it) and so we generally suppress it, using a $\cdot$ instead, as in the conjecture below.

\begin{Conjecture}\label{prop:Q0}
For $a\ge -1$, the radius of convergence of 
$\Q(a,\cdot)$ is
\beq\label{radius}
\rho_\Q(a)= \left\{
\begin{array}{ll}
\displaystyle \frac 1 {(2+\sqrt{2+2a})^2} &\hbox{if } a\ge -1/2,
\\
\\
\displaystyle -\frac{{a}}{ 2 (a-1)^2} & \hbox{if } a \in [-1,-1/2].
\end{array} \right.
\eeq
\end{Conjecture}
\begin{Conjecture}\label{conj:conv}
The series $Q'_2(a, u):= \frac{\partial Q}{\partial u}(a,u)$ is
convergent at $u=\rho_Q(a)$ for $a\ge -1/3$.
\end{Conjecture}

We shall only use the first part of Conjecture~\ref{prop:Q0} (in fact,
for $a\ge -1/3$ only).   
The analytic techniques
of~\cite{kurkova-raschel-nature,fayolle-raschel} may open a way to its
proof. In fact,  
Kilian Raschel was able to predict these values for the radius from a
(not yet rigorous) application of these techniques. 
We have also checked this conjecture numerically, using the ratio test
(Figure~\ref{fig:radius}, left). Regarding Conjecture~\ref{conj:conv}, if we
assume that
$$
q_n(a):=[u^n]Q(a,u) \sim \kappa(u) \rho_Q(a)^{-n} n^{\gamma(a)}
$$
we would need $\gamma(a)<-2$. This is in good agreement with the
estimates of $\gamma(a)$ shown in Figure~\ref{fig:radius}, right. It
is likely that $Q'_2(a, \rho_Q(a))$ does not converge at $a=-1/2$.

\begin{figure}[ht]
\begin{center}
\includegraphics[width=5cm]{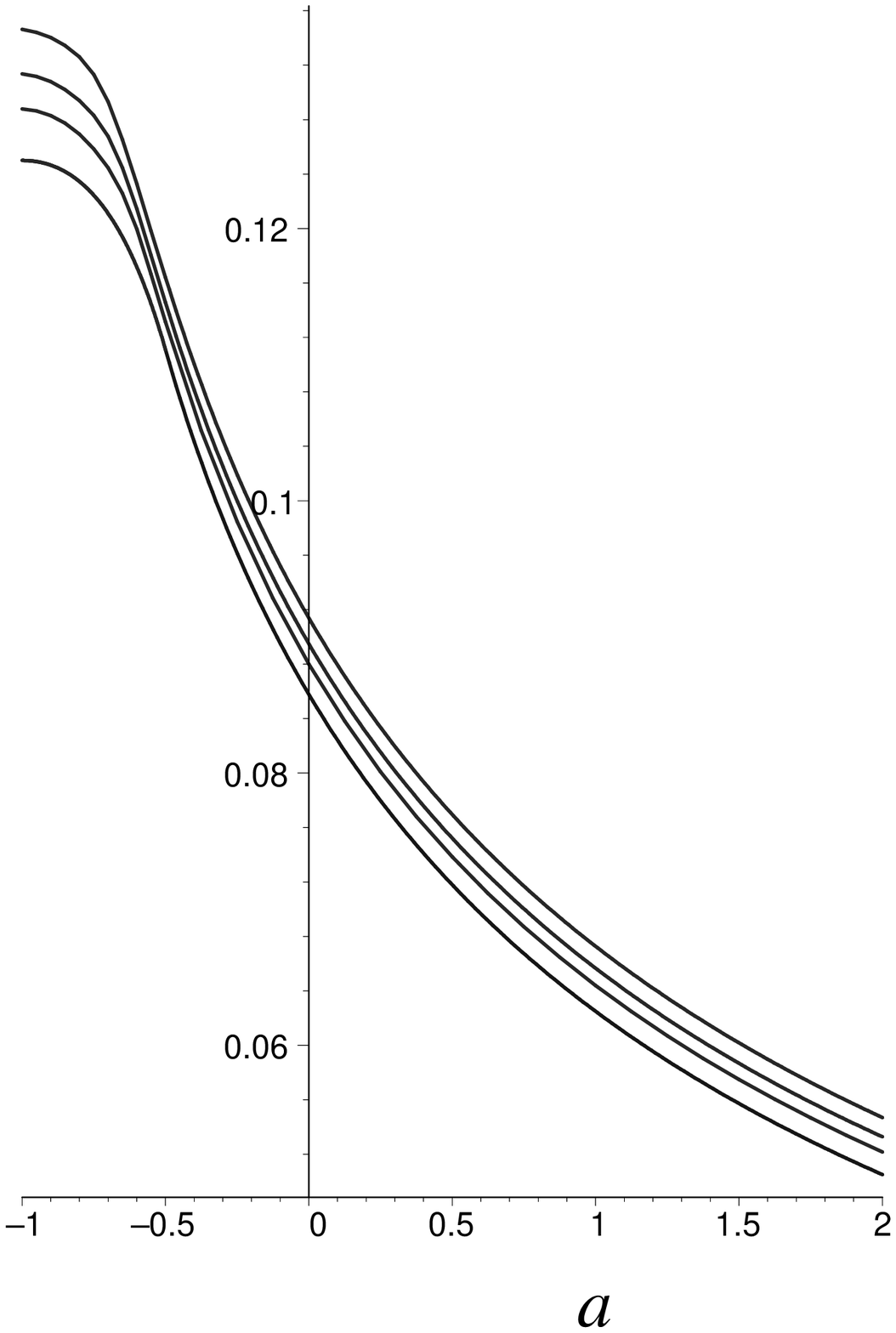}
\hskip 20mm\includegraphics[width=5cm]{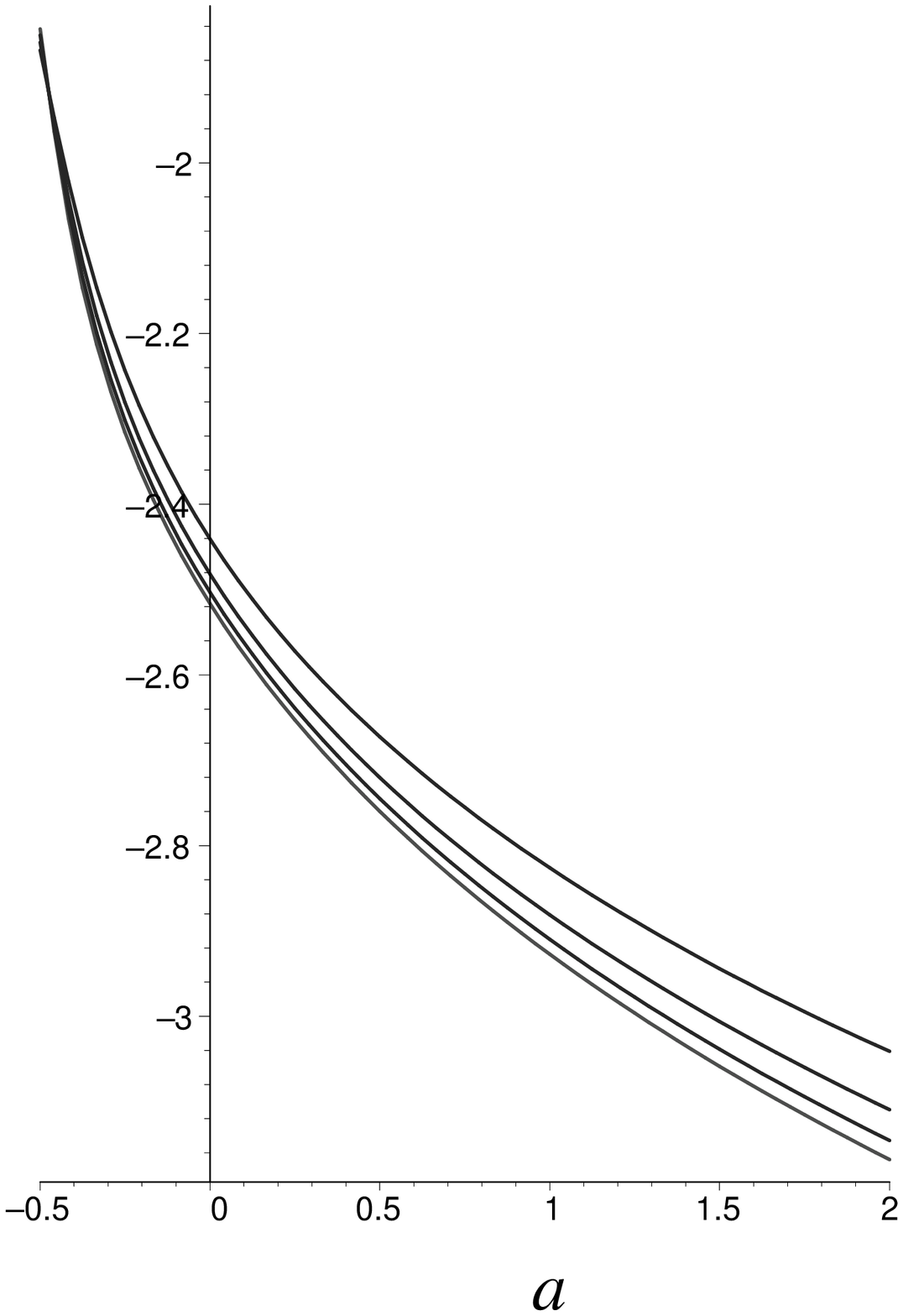}
\caption{Left: The three top curves show the ratio $q_{n-1}(a)/
  q_n(a)$, where $q_n(a)=[u^n] Q(a,u)$, for $n=40$,   $60$ and  $100$
  and $a\in [-1,2]$. The curves 
  seem to accumulate, as 
  $n$ grows, on the conjectured radius (bottom curve). Right: the
  curves $n^2(1- q_{n-1}(a)q_{n+1}(a)/q_n(a)^2)$, shown for $n=40$,
  $60$, $80$ and $100$, provide estimates for the exponent $\gamma(a)$.}
\label{fig:radius}
\end{center}
\end{figure}

% \medskip
In the following subsections, we gather more
 evidence for these
conjectures. In particular, we prove  Conjectures~\ref{conj:a+1}
and~\ref{prop:Q0} for  general
loops (Section~\ref{sec:a+1}), and for loops confined to the upper
half plane (Section~\ref{sec:half-plane}). Note that Conjecture~\ref{conj:conv}
does not hold for these more general loops.
The fact that Conjecture~\ref{prop:Q0}
holds for general loops and
half plane loops is reminiscent of a recent result according to which
the growth constant of (unweighted) loops confined to a wedge is
independent of this wedge~\cite[Sec.~1.5]{denisov-wachtel}.

We also prove the %three
conjectures for $a=1$ and $a=-1$. In the latter case Conjecture~\ref{conj:a+1} then
simply  means that $\Q(-1,u)$ has non-negative
coefficients; these coefficients are in fact very nice, see
Proposition~\ref{prop:Qm1}.
It is interesting to refine the 
enumeration by taking
into account  the number of 
\EE\ steps with a new variable $s$. This gives rise to a
generating function denoted $\Q(a,s,u)$. In fact, we have the following refined conjecture.

\begin{Conjecture}\label{conj:a+1-ref}
The series $\Q(a,s,u)$ %counting quarter plane loops 
is
$(a+1)${-positive}, as well as the series $\Q^\bullet(a,s,u)$
counting  primitive quarter plane loops.
\end{Conjecture}
This has been checked on a computer up to half-length $n=40$,  using the
following refinement of the functional equation~\eqref{eq:Qa}:
\begin{multline}
  \label{eq:Qa-ref}
 (1-u(sx+s\bx+y+\by)-u^2s(a-1)(x\by+y\bx))\Qa(x,y)=\\
1-u\by (1+usx(a-1))\Qa(x,0)-us\bx (1+uy(a-1))\Qa(0,y) .
\end{multline}
This equation characterises the series $\Qa(x,y)\equiv \Qa(a,s,u;x,y)$ that
counts quarter plane walks by \NW\ and \ES\  corners ($a$), horizontal
steps ($s$), total length ($u$) and coordinates of the endpoint ($x,y$). In
particular, the above defined series $Q(a,s,u)$ is $\Qa(a,\sqrt s,
\sqrt u; 0,0)$. Of course, $\Q(a,s,u)=1/(1-\Q^\bullet(a,s,u))$, and
the second part of  Conjecture~\ref{conj:a+1-ref} implies the first 
one. We discuss in Section~\ref{sec:final} further investigations on
these conjectures.

Before we embark on our results, we want to report an observation,
due (independently) to Olivier Bernardi 
and Julien Courtiel, which might be useful to prove the above
conjectures. It tells that the pair $(\NW,\ES)$ can be replaced by other
pairs of corners. Let us say that two words on the alphabet $\{\NN,\SS,\EE,\WW\}$ are \emm
shuffle-equivalent,  if they have the same  projections on $\{\NN,\SS\}$,
and also on $\{\EE,\WW\}$. For instance, the words {\sf NEWSSWWNES} and
{\sf ENWWSSNWES} are shuffle-equivalent.  A \emm shuffle class, is an
equivalence class for this relation.

\begin{Proposition}\label{cor:equidistributed}
  There exists an involution $\Phi$ on square lattice walks that
  exchanges the number of \NW\ and \WN\   factors,  fixes
  the number of \ES\  and \SE\ factors and acts inside shuffle classes.

 Consequently, in every shuffle class, the following bi-statistics of
 corners are   equidistributed: $(\NW,\ES)$, $(\WN,\ES)$, $(\WN,\SE)$ and
 $(\NW,\SE)$. 
\end{Proposition}
\begin{proof}
  To construct $\Phi(w)$, read backwards every maximal factor of $w$ consisting of \NN\ and \WW\ steps:
  this transforms every \NW\ factor into a \WN\ factor, and vice-versa. For
  instance, the word {\sf ENWWSSNWES} becomes {\sf EWWNSSWNES}. This is clearly an
  involution, which satisfies the announced properties.

 The equidistribution of $(\NW,\ES)$ and $(\WN,\ES)$ follows. The
 equivalence with the other pairs follows from simple variants of the
  involution $\Phi$. 
\end{proof}

%============================================================
\subsection{Some results on quarter plane loops}
%============================================================
\begin{Proposition}
%[{\bf The series $\boldsymbol {Q(a,u)}$ at $\boldsymbol {a=1}$ and $\boldsymbol {a=-1}$}]
\label{prop:Qm1}
 The series $Q(1,u)$ counting quarter plane loops by the
 half-length is
$$
Q(1,u)= 
%\sum_{n\ge 0} u^n\sum_{i=0} ^n {2n \choose {2i}} C_i C_{n-i}
\sum_{i, j \ge 0}  {2i+2j \choose {2i}} C_i C_{j}u^{i+j}
=\sum_{n\ge 0} 
 C_nC_{n+1}
%\frac 1{ (2n+1)(2n+4) } {2n+2\choose n+1} ^2
u^n ,
$$
% mbm "th" in exponent below
where $C_i={2i\choose i}/(i+1)$ is the $i^{\mbox{\scriptsize th}}$ Catalan number. This can be refined by taking
into account  the number of \EE\ steps (with a variable $s$):
\beq\label{Q1-ref}
Q(1,s,u)=
% \sum_{n\ge 0} u^n\sum_{i=0} ^n {2n \choose {2i}} C_i C_{n-i}s^i.
\sum_{i, j \ge 0}  {2i+2j \choose {2i}} C_i C_{j} s^i u^{i+j}.
\eeq
The value of $Q(a,s, u)$ at $a=-1$ is just as remarkable:
%$$Q(-1,u)= \sum_{n\ge 0} u^n \sum_{i=0} ^n {n \choose i} C_i C_{n-i}.$$More precisely, if $s$ keeps track of half the number of horizontal steps,
\beq\label{expr-Qm1}
Q(-1,s,u)= 
%\sum_{n\ge 0} u^n \sum_{i=0} ^n {n \choose i} C_i C_{n-i}s^{i}.
\sum_{i, j \ge 0} {i+j \choose {i}} C_i C_{j} s^i  u^{i+j}.
\eeq
\end{Proposition}
In particular, the coefficients of $Q(-1,s,u)$ are non-negative, which
is a very weak form of Conjecture~\ref{conj:a+1-ref}.
\begin{proof}
When $a=1$, we do not take corners into account.  The results dealing
with $Q(1,u)$ and $Q(1,s,u)$ are well-known and 
  easy to prove: it suffices to observe that a quarter plane loop
  is obtained by shuffling two Dyck paths, one on the alphabet
  $\{\NN,\SS\}$ and the other on the alphabet $\{\EE,\WW\}$. Since there are
  $C_i$ Dyck paths of length $2i$, this gives directly the expression
  of $Q(1,s,u)$, and hence the first expression of $Q(1,u)$. The
  second one follows using the Chu-Vandermonde summation. See
  also~\cite{cori-baxter} for a (recursive) bijective proof,
  and~\cite{bernardi-tree-rooted} for a  non-recursive one.

For the case $a=-1$, we  work from the
functional equation~\eqref{eq:Qa-ref}. The proof, inspired by recent
progress of general quarter plane walks~\cite{mbm-mishna},  is a bit
long. It is given in Appendix~\ref{app:m1}.
\end{proof}

%\noindent
\textbf{Remark}. The above expressions imply that $Q(a,u)$ is
\emm D-finite, for $a=1$ and $a=-1$. That is, it satisfies a linear
differential equation (LDE) in $u$ with polynomials in $\qs[u]$. We suspect
that  $Q(a,u)$ does not satisfy
any LDE with coefficients in $\qs[a,u]$. Using the Maple package
gfun~\cite{gfun}, we have tried in vain to guess an LDE for $Q(0,u)$ from the
first 300 coefficients.

%============================================================
Let us now discuss the radius of convergence of $Q(a, \cdot)$. We
begin with a simple lemma, which is often used in a
statistical physics context.
\begin{Lemma}\label{lem:cont}
  Let $F(a,u)=\sum_{n\ge 0} f_n(a)u^n$ be a \fps\ in $u$ with
  coefficients in $\ns[a]$, such that $f_n(a)$ has degree at most $n$.
Assume that $F$ is not a polynomial.

For $a\ge 0$, let $\rho(a)$ be the radius of convergence of the series
$F(a, \cdot)$.  Then 
%for $a>0$, and is a 
 $\rho$ is a non-increasing  function on $[0,+\infty)$, which is finite
 and continuous on  $(0,+\infty)$.
\end{Lemma}
\begin{proof}
  Since $F$ is not a polynomial, there exist infinitely many $n$ such
  that $f_n(a)\not = 0$. In this case, we have, for $a>0$:
$$
f_n(a)\ge \min(1, a^n).
$$
This shows that $\rho(a)\le \max(1, a^{-1})$, and is, in particular,
finite for $a>0$.

That $\rho(a)$ is non-increasing comes from the fact that $f_n(a)$ is
non-decreasing. Finally, if $0<a\le a'$, we have
$$
f_n(a') \le \left(\frac {a'}a\right)^n f_n(a)
$$
(since $f_n$ has degree at most $n$), which gives
$$
\rho(a') \ge \frac a {a'}\, \rho(a),
$$
and, together with $\rho(a') \le \rho(a)$, establishes the continuity
of $\rho$ in $(0, +\infty)$.
\end{proof}

\begin{Proposition}%[\bf{On the radius of convergence of $\boldsymbol {Q}$}]
\label{prop:radiusQ}
  For fixed $a$, let $\rho_\Q(a)$ be the radius of convergence of
  $\Q(a, \cdot)$. Then
$$
\rho_Q(-1)= \frac 1  8, \quad %\hbox{and}
 \quad \rho_Q(1)= \frac 1 {16},
$$
and $\rho$ is a non-increasing function on $[0, +\infty)$,  continuous
on $(0, +\infty)$. Moreover, for $a\ge 0$, 
\beq\label{radius-bound}
\rho_Q(a) \ge \frac 1 {(2+\sqrt{2+2a})^2}.
\eeq
The series $\Q'_2(a, \rho_Q(a))$ converges for $a=-1$ and $a=1$.

\smallskip
 If $\Q(a,u)$ is $(a+1)$-positive,  then $\rho_\Q$
  is non-increasing  on $[-1, +\infty)$ and continuous on
  $(-1, +\infty)$. 
\end{Proposition}

% \noindent 
Note that Conjecture~\ref{prop:Q0} says that the
bound~\eqref{radius-bound} is tight. 
\begin{proof}
  The first two results follow from the explicit expressions of
  Proposition~\ref{prop:Qm1}. At $a=1$, we simply apply Stirling's
  formula to the second expression of $Q(1,u)$ to obtain the radius. More
  precisely, we find
$$
[u^n] Q(1,u) \sim \frac 4 \pi 16^n n^{-3}.
$$
At $a=-1$, we have to determine the asymptotic behaviour of a sum of
positive terms. We use the approach described %for instance
in~\cite[Section~3]{bender}, and find
$$
[u^n] Q(-1,u) \sim \frac{1}\pi 8^{n+1}n^{-3}.
$$
These estimates imply the convergence of $Q'_2(a,\rho_Q(a))$ at $a=1$
and $a=-1$. 

Since a walk of half-length $n$ has at most $(n-1)$ \NW\ or \ES\
factors, the properties of $\rho_Q$ on $[0, +\infty)$ are a direct
application of Lemma~\ref{lem:cont}. 

The lower bound of $\rho_Q(a)$ for $a\ge 0$ follows from the fact that
$\Q(a,u ^2)$ is dominated by the series  counting \emm all,
square lattice walks by the length and the number of \NW\ and \ES\ 
corners. This series is easily seen to be
$$
\frac 1 {1-4u -2{u}^{2}(a-1) },
$$
 and its radius is
$ 1/( {2+\sqrt{2+2a} })$. See the proofs of Propositions~\ref{prop:general}
 and~\ref{prop:radius-general} for details. This shows that the radius
 of $\Q(a,u ^2)$ is at least $ 1/( {2+\sqrt{2+2a} })$, which is
 equivalent to~\eqref{radius-bound}.
\end{proof}

%===================================================
\subsection{General loops: Generating functions}
\label{sec:general}
%===================================================

We  now address the enumeration of general loops according to the
number of \ES\  and \NW\ corners.  Their generating function can be obtained by two
successive coefficient extractions in a rational \gf. This will allow
us to prove that Conjectures~\ref{conj:a+1} ($(a+1)$-positivity) and~\ref{prop:Q0}
(radius of convergence) hold for general loops.

\begin{Proposition}\label{prop:general}
  The generating function $\Ra(a,s,t;x,y)$ counting  square lattice walks  by the
  number of horizontal steps ($s$), the %length
number of vertical steps $(t)$,  the number of \ES\ 
  and \NW\ corners ($a$) and the   coordinates of the endpoint ($x,y$)
  is rational, and given by
%$$\Ra(a,s,t;x,y)=  \frac \A{1-ty\B-t\by \C},$$where $\A, \B$ and $\C$
%are given in Lemma~\ref{prop:ABC}. Equivalently,
\beq\label{W}
\Ra(a,s,t;x,y)=
%\frac 1{1-t(x+\bx+y+\by)-t^2(a-1)(x\by+y\bx)}.
%
\frac 1{1-{ {s \left( {x}+\bx \right) }}-{t { \left( {y}+\by
 \right) }} -{st \left( a-1 \right) { \left( {x}\by+\bx{y} \right) }}}.
\eeq
The generating function that only counts walks ending on the $x$-axis is algebraic,
and given by
\begin{align}
%\begin{eqnarray*}
  \Ra_{-,0}(a,s,t;x) &:=[y ^0]\Ra(a,s,t;x,y) \nonumber \\
%&=& \frac \A{\sqrt{1-4t^2 \B  \C}}\\
& \label{W0}=\frac 1 {\sqrt{(1-s(x+\bx))^2-4t^2(1+sx(a-1))(1+s\bx(a-1))}}.
\end{align} The generating function that only counts loops is D-finite, and given by
\beq\label{g-00}
\Ra_{0,0}(a,s,t):=[x^0y ^0]\Ra(a,s,t;x,y)= \sum_{j\ge 0} {2j\choose j}
t^{2j} \Ra_{0,0,j}(a,s)
\eeq
where 
\beq\label{W00j}
\Ra_{0,0,j}(a,s)= [x^0]
\frac{(1+sx(a-1))^j(1+s\bx(a-1))^j}{(1-s(x+\bx))^{2j+1}}. 
\eeq
\begin{comment}
The generating function of the series $\Ra_{0,0,j}(a,s)$ is
\beq\label{G0expr}
A(a,s,t):= \sum_{j\ge 0} t^{2j}\Ra_{0,0,j}(a,s)=% [x^0]\frac \A{1-t^2\B\C}=
\sqrt{\frac{P+2(a+s^2(a-1)^2)\sqrt{P_{++}P_{-+}P_{+-}P_{--}}}
{ \left( 4a +(a-1)^2(4s^2+t^2)\right)P_{++}P_{-+}P_{+-}P_{--}}},
\eeq
with
$$
P_{++}\equiv P_{++}(s,t)= 1+2s+t+st(1-a),
$$
$$
P_{-+}= P_{++}(-s,t), \quad P_{+-}= P_{++}(s,-t),  \quad P_{--}= P_{++}(-s,-t),
$$
and
$$
P=  2a+2 \left( {a}^{2}-6 a+1 \right) {s}^{2}+ \left( {a}^{2}+1 \right) {t}^{2}+2 \left( a-1 \right) ^{3}{s}^{2}{t}^{2}-8 \left( a-1 \right) ^{2}{
s}^{4}
+2 \left( a-1 \right) ^{4}{s}^{4}{t}^{2}.
$$
\end{comment}
The generating function 
\beq\label{A-def}
A(a,s,t):=\sum_{j\ge 0} t^{2j}\Ra_{0,0,j}(a,s)
\eeq
 is biquadratic, and can be written as 
\beq\label{Aast}
  A(a,s,t)=
%\frac{1+t^2(1-a^2)T}{\left(1-t-2t(a+1)T-t^2(a^2-1)T\right)\left(1+t+2t(a+1)T-t^2(a^2-1)T\right)}\\
\frac{1+t^2(1-a^2)T}
{\left(1-t^2(a^2-1)T   \right)^2 - t^2(1+2(a+1)T)^2}
%\\ \times 
\ \sqrt{\frac{1+4T-t^2(a-1)^2T}{1-t^2(a-1)^2T}},
\eeq%\end{multline} 
where $T\equiv T(a,s,t)$ is the unique series in
  $\qs[a][[s,t]]$  that satisfies   $T(a,0,t)=0$ and
\beq\label{eq:T}
T = s^2\, \frac{1 + 4T - t^2 (a-1)^2 T}{1-t ^2-t^2(a+1)^2T}.
\eeq
\end{Proposition}
% \noindent 
As explained in the proof below,  finding the
expression of $\Ra(a,s,t;x,y)$ is simple, and then the rest of the proposition
follows using two consecutive coefficient extractions. We will however
give at the end of this subsection an alternative, more combinatorial proof
of~(\ref{W}-\ref{W00j}), which explains in particular the factor
${2j\choose j}$ occurring in~\eqref{g-00}.
\begin{proof}
To establish the expression of $\Ra(a,s,t;x,y)$, we use the
step-by-step construction of walks that was used in Theorem~\ref{thm:eqs} to
establish an equation for $\Qa(x,y)$. The argument is simplified by
the fact that we have no boundary. That is, the terms
$\Qa(x,0)$ and $\Qa(0,y)$ that occur in~\eqref{eq:Qa} disappear, and one
obtains~\eqref{W}.

There are several ways to obtain the expression~\eqref{W0} for walks
ending on the $x$-axis. One can for instance write a system of
algebraic equations by decomposing walks at their first visit on the
$x$-axis, in the spirit of what one usually does to count Dyck paths
(see e.g.~\cite[Sec.~V.4]{flajolet-sedgewick}).  We can also
expand~\eqref{W} directly in $y$:
%\begin{eqnarray}
%  [y^0]\Ra(a,s,t;x,y)&=& [y^0] \frac 1{\left(1-s(x+\bx)\right) 
%\left( 1-\frac{ty(1+s(a-1)\bx)}{1-s(x+\bx)}
%-\frac{t\by(1+s(a-1)x)}{1-s(x+\bx)}\right)}\nonumber
%\\
%&=&  [y^0]\sum_{i,j\ge 0} {i+j \choose i}
%\frac{t^{i+j}y^{j-i}(1+s(a-1)\bx)^j(1+s(a-1)x)^i}{\left(1-s(x+\bx)\right)
%  ^{i+j+1}}\nonumber
%\\
%&=&\sum_{j\ge 0} {2j \choose j}
%\frac{t^{2j}(1+s(a-1)\bx)^j(1+s(a-1)x)^j}{\left(1-s(x+\bx)\right)
%  ^{2j+1}},\label{W0-exp}
%\end{eqnarray}
\begin{multline}
\label{W0-exp}
  [y^0]\Ra(a,s,t;x,y) =  [y^0] \frac 1{\left(1-s(x+\bx)\right) 
\left( 1-\frac{ty(1+s(a-1)\bx)}{1-s(x+\bx)}
-\frac{t\by(1+s(a-1)x)}{1-s(x+\bx)}\right)}
\\
=  [y^0]\sum_{i,j\ge 0} {i+j \choose i}
\frac{t^{i+j}y^{j-i}(1+s(a-1)\bx)^j(1+s(a-1)x)^i}{\left(1-s(x+\bx)\right)
  ^{i+j+1}}\\
=\sum_{j\ge 0} {2j \choose j}
\frac{t^{2j}(1+s(a-1)\bx)^j(1+s(a-1)x)^j}{\left(1-s(x+\bx)\right)
  ^{2j+1}},
\end{multline}
which is equivalent to~\eqref{W0} since $\sum_{j\ge 0}  {2j \choose j}v^j=(1-4v)^{-1/2}$.
\begin{comment}
Even though this is not vital, we find it convenient to have one main
length variable $u$, that is, to replace $t$ by $u$ and $s$ by
$su$. The denominator of $\Ra(a,su,u;x,y)$ is now a Laurent polynomial
in $y$ of degree 1 and valuation 1. It has two roots, denoted $Y_1$
and $Y_2$, which can of course be written explicitly. They are Laurent
series in $u$ with coefficients in $\qs[a,s,x,\bx]$.
One root, say $Y_1$, is actually a \fps\ in $u$ and vanishes at $u=0$:
$$
Y_1= u+ s(ax+\bx) u^2+o(u^3),
$$
while the other involves a negative power in $u$:
$$
Y_2=u^{-1} -s(x+a\bx) + O(u).
$$
The partial fraction expansion in $y$ of $\Ra(a,su,u;x,y)$ can be written
\beq\label{Ru}
\Ra(a,su,u;x,y)= \frac 1{\sqrt \Delta} \left( \frac 1
  {1-y/Y_2}+\frac{\by Y_1}{1-\by Y_1}\right),
\eeq
where $\Delta$ is the discriminant that occurs in $Y_1$ and $Y_2$:
$$
\Delta={1-
{ {su \left( {x}+\bx \right) }}-{u { \left( {y}+\by
 \right) }}
 -{su^2\left( a-1 \right) { \left( {x}\by+\bx{y} \right) }}}.
$$
Since $1/Y_2$ are $Y_1$ are multiples of $u$, we can now read off
from~\eqref{Ru} the coefficient of $y^0$. This gives 
$$
[y^0] \Ra(a,su,u;x,y)= \frac 1{\sqrt \Delta} ,
$$
which is equivalent to~\eqref{W0}.\end{comment}

Let us now count loops. Equations~\eqref{g-00} and~\eqref{W00j} are easily
obtained by extracting the coefficient of $x^0$ in~\eqref{W0-exp}.
We now want to obtain an expression for 
$$%\beq\label{Ras}
A(a,s,t):= \sum_{j\ge 0} t^{2j}\Ra_{0,0,j}
%= [x^0]\frac \A{1-t^2\B\C}.
= [x^0] R(a,s,t;x),
$$%\eeq
where 
$$
R(a,s,t;x)=\frac{1-s(x+\bx)}{(1-s(x+\bx))^2-t^2(1+sx(a-1))(1+s\bx(a-1))}.
$$
As in Appendix~\ref{app:m1}, we want to extract a coefficient in a
rational fraction (the above series $R$, specialised to $a=-1$, is in
fact related to the series~\eqref{R-def} considered in the appendix). Even though
this is not vital, we find it convenient 
to have one main 
length variable $u$, that is, to replace $s$ by $su$ and $t$ by
$u$.  The denominator of $R(a,su,u;x)$ is a Laurent polynomial in $x$, symmetric in $x$ and
$\bx$, of degree 2. It has four roots, which  are Laurent 
series in $u$ with coefficients in $\qs(\sqrt a, s)$ (we refer  to~\cite[Chapter~6]{stanley-vol2} for generalities on
solutions of polynomial equations with  coefficients in $\GK(u)$, for
a field $\GK$ of characteristic $0$). Two of the roots,
denoted $X_1$ and $X_2$,  are
actually power series in $u$, and they vanish at $u=0$:
$$
X_{1,2}=
su\pm \sqrt a s{u}^{2}+\frac s 2  \left( a +1+2{s}^{2} \right)
{u}^{3}
% \pm\,{\frac {s \left(1+6\,a +{a}^{2}+4\,{s}^{2} +16\,{a}{s}^{2}+
% 4\,{a}^{2}{s}^{2} \right) }{8\sqrt a }}{u}^{4}
+O\left( {u}^{4} \right) .
$$
The other two are $\bX_1:=1/X_1$ and  $\bX_2:=1/X_2$.
%, and they read $1/(us)+O(1)$.  
We now perform a partial fraction expansion of
$R(a,su,u;x)$ with respect to $x$:
\begin{eqnarray}
  R(a,su,u;x) &=&
\frac{X_1X_2(1-su(x+\bx))}{s^2u^2(1-xX_1)(1-xX_2)(1-\bx X_1)(1-\bx X_2)}
\nonumber
\\
&=&\frac{\alpha_1}{1-xX_1} + \frac{\alpha_2}{1-xX_2} +
\frac{\alpha_1\bx X_1}{1-\bx X_1}  +
\frac{\alpha_2\bx X_2}{1-\bx X_2} 
\label{R-parfrac}
\end{eqnarray}
where
$$
\alpha_1=\frac{ 1-su(X_1+\bX_1)}{s^2u^2
  (X_1-\bX_1)(1-\bX_1X_2)(X_1-\bX_2)}
$$
and symmetrically for $\alpha_2$. Since $X_{1}$ and $X_2$ are
multiples of $u$, we can read off from~\eqref{R-parfrac} the coefficient of $x^0$ in $R$:
\begin{align*}
A(a,su,u) &=[x^0]R(a,su,u;x) = {\alpha_1} + {\alpha_2} \\
&=
\frac{1-2su(X_1+X_2)+X_1X_2}{s^2u^2(X_1-\bX_1)(X_2-\bX_2) (1-X_1X_2)}.
\end{align*}
We finally eliminate $X_1$ and $X_2$ using the algebraic equations
they satisfy (recall that they cancel the denominator of
$R(a,su,u;x)$). This gives an algebraic equation for $A(a,su,u)$. This 
equation has two distinct factors, both of degree 2 in $A^2$. Only one
of these factors has some roots in $\qs[a,s][[u]]$ (where we expect
$A(a,su,u)$ to be): this factor is the minimal algebraic equation
satisfied by $A(a,su,u)$. After replacing $s$ by $s/u$ and then $u$ by
$t$, we obtain the minimal algebraic equation satisfied by $A(a,s,t)$.

 This equation only involves even powers of $s$ and $t$. Let us now replace $s^2$ by its
expression in terms of $a,t$ and $T$ derived from~\eqref{eq:T}. The
resulting equation factors into two terms, each of   
degree one in $A(a,s,t)^2$. Only one of these terms has a solution in
$\qs[a][[s,t]]$ with constant term~1, and solving it for $A(a,s,t)$
gives~\eqref{Aast}. 
\end{proof}

Another proof of~(\ref{W}-\ref{W00j}) can be given by considering another
pair of corners, as allowed to us by
Proposition~\ref{cor:equidistributed}. The following proposition explains
in particular the factor ${2j\choose j}$ occurring in~\eqref{g-00}.
\begin{Proposition}\label{prop:ABC}
  Let $v$ be a word on $\{\NN,\SS\}$. The generating function of walks whose vertical
  projection is $v$, counted by the number of horizontal steps ($s$),
  the abscissa of the endpoint ($x$) and the number of \NW\ and \SE\
  factors (or any equivalent statistic  from
  Proposition~\ref{cor:equidistributed}; variable $a$) 
  only depends on $|v|_\NN$ and $|v|_\SS$. Its value is
\beq\label{v-fixed}
\A \B^{|v|_\NN}\C^{|v|_\SS},
\eeq
where
$$
\A=\frac 1 {1-s(x+\bx)}, \quad \B=
% 1+ \frac{a\bx+x}{1-s(x+\bx)} s
\frac{1+ s\bx(a-1)}{1-s(x+\bx)}
\quad \hbox{and} \quad  \C= 
%1+ \frac{ax+\bx}{1-s(x+\bx)} s.
\frac{1+ sx(a-1)}{1-s(x+\bx)}.
$$
\end{Proposition}
\begin{proof}
  Write $v=v_1 \cdots v_n$. The walks we want to count read $w_0v_1
  w_1 \cdots v_n w_n$, where the $w_i$'s are words on the alphabet
  $\{\EE, \WW\}$. Observe that \NW\ and \SE\ factors can only be created just
  after a \NN\ or \SS\ step. In particular, the contribution of $w_0$ is
  $\A$. After a \NN\ step $v_i$, a \NW\ factor is created if
  and only if $w_i$ begins with the letter \WW; this shows that the
  contribution of $w_i$ is 
$$
1+s\, \frac{a\bx+x}{1-s(x+\bx)} ,
$$
which is precisely $\B$. Similarly, the factors $w_i$ following
  a step $v_i=\SS$ contribute the series $\C$. 
\end{proof}
% \noindent
\textbf{Application}. One can now rederive
the expressions~\eqref{W} and~\eqref{W0} of $\Ra(a,s,t;x,y)$
and $\Ra_{-,0}(a,s,t;x,y)$  by summing~\eqref{v-fixed},
respectively over all walks on $\{\NN, \SS\}$ and over walks on $\{\NN, \SS\}$
ending at ordinate $0$. Moreover, the series $A(a,s,t)$ given
by~\eqref{Aast} can now be understood as the generating function of loops which have
vertical projection {\sf NSNSNS...}

%============================================
\subsection{General loops: $\boldsymbol {(a+1)}$-positivity and radius
  of convergence}
\label{sec:a+1}
%============================================
\begin{Proposition}\label{prop:a+1-R}
  The series $\Ra_ {0,0}(a,s,t)$ that counts general loops is
  $(a+1)$-positive: for $i, j \ge 
  0$, the coefficient of $s^i t^j$ in this series is a polynomial in
  $(a+1)$ with non-negative coefficients.

Moreover,
$$
\Ra_{0,0}(1,s,t)= \sum_{i, j \ge 0} s^{2i} t^{2j}{2i+2j \choose 2i}
{2i \choose i } {2j\choose j},
$$
while
\beq\label{expr-Rm1}
\Ra_{0,0}(-1,s,t)= \sum_{i, j \ge 0} s^{2i} t^{2j}{i+j \choose i}
{2i \choose i } {2j\choose j}.
\eeq
When $s=t=u$,
$$
\Ra_{0,0}(1,u,u)= \sum_{n \ge 0} u^{2n}{2n\choose n}^2.
$$
\end{Proposition}
\begin{proof}
  By~\eqref{g-00}, the first statement means that for all $j$, the
  series $\Ra_{0,0,j}(a,s)$ is $(a+1)$-positive, or, equivalently,
  that the quartic series given in~\eqref{Aast} is
  $(a+1)$-positive.  Let us first prove that $T$, given by~\eqref{eq:T}, is $(a+1)$-positive. This
follows by observing that~\eqref{eq:T} can be written as
$$
T= s^2\left( 
\frac{4t^2(a+1)T}{1-t^2-t^2(a+1)^2T} + 
\frac{4T} {1- t^2(a+1)^2 \frac T{1-t^2}} + \frac 1 {1- \frac {t^2}{1-
    t^2(a+1)^2 T}}
\right).
$$
Indeed, this equation is equivalent to a recurrence relation defining the coefficient of $s^{2i}t^{2j}$ in $T$ by
recurrence on $i+j$. 
This recurrence expresses this coefficient, denoted $T_{i,j}$, as a
polynomial in $(a+1)$ and the $T_{k,\ell}$ for $k+\ell <i+j$  \emm with non-negative
coefficients,, and thus
proves $(a+1)$-positivity of $T$.

Now the rational factor in~\eqref{Aast}, one converted in partial
fractions of $T$, reads 
$$
\frac 1 {2\left(1-t-2t(a+1)T-t^2(a^2-1)T\right)}+\frac 1{
2\left(1+t+2t(a+1)T-t^2(a^2-1)T\right)}.
$$
%\frac 1 {(1-t) \left( 1- t(a+1) T \left(2+ \frac      {(a+1)t}{1-t}\right)\right)} +\frac 1 {(1+t) \left( 1+ t(a+1) T \left(2- \frac      {(a+1)t}{16t}\right)\right)} .$$
As a rational series in $t$ and $T$ (with coefficients in $\qs[a]$),
it is thus the even part in $t$ of the series  
$$
\frac 1 {1-t-2t(a+1)T-t^2(a^2-1)T}
= \displaystyle  \frac 1 {1-t}\times \frac 1 { 1- t(a+1) T \left(2+ \frac
      {t(a+1)}{1-t}\right)} ,
$$
which can be expanded in $t$, $T$ and $(a+1)$ with non-negative
coefficients. This proves the  $(a+1)$-positivity of the rational
factor in~\eqref{Aast}.  Finally, using the equation~\eqref{eq:T}
satisfied by $T$, the square root factor in~\eqref{Aast} can be
written
$$
\displaystyle \frac 1{\displaystyle \sqrt{1- \frac {4s^2}{1-t^2-t^2(a+1)^2T}}},
$$
which is an $(a+1)$-positive series in $s$, $t$ and $T$, and thus an
$(a+1)$-positive series in $s$ and $t$. This proves finally that the series $A(a,s,t)$ is $(a+1)$-positive, as well as the
generating function $\Ra_{0,0}(a,s,t)$ of general loops.

\medskip

The rest of the proof is now easier. The expression of
$\Ra_{0,0}(1,s,t)$ follows from the fact that square lattice
loops are just shuffles of one-dimensional loops. The expression of
$\Ra_{0,0}(1,u,u)$  follows by the Chu-Vandermonde
identity. Alternatively, it can be proved combinatorially by
projecting loops on the diagonals $x= \pm y$. The
expression of $\Ra_{0,0}(-1,s,t)$ is of course more surprising, but it comes
out easily from the work we have already done. By comparing the
expressions~\eqref{g-00} of $\Ra_{0,0}$ 
and~\eqref{A-def} of $A(a,s,t)$, we see that what we have to prove reads
\begin{eqnarray*}
A(-1,s,t)&=& \sum_{i,j\ge 0} s^{2i} t^{2j} {i+j\choose i} {2i
  \choose i}
\\
&=& \sum_{i\ge 0} \frac{s^{2i}}{(1-t^2)^{i+1}} {2i  \choose i}
\\
&=&\frac 1 {\sqrt{(1-t^2)(1-4s^2-t^2)}}.
\end{eqnarray*}
This is readily proved by specializing~\eqref{eq:T} and~\eqref{Aast}
to $a=-1$. In particular, $T$ becomes rational for this value of $a$.
\begin{comment}\begin{eqnarray*}
\sqrt{\frac{P+2(-1+4s^2)\sqrt{P_{++}P_{-+}P_{+-}P_{--}}}
{ \left( -4 +4(4s^2+t^2)\right)P_{++}P_{-+}P_{+-}P_{--}}}
&=&\sum_{j\ge 0} t^ {2j} \sum_{i\ge 0} s^{2i} {i+j \choose i} {2i
  \choose i}
\\
&=& \sum_{i\ge 0} \frac{s^{2i}}{(1-t^2)^{i+1}} {2i  \choose i}
\\
&=&\frac 1 {\sqrt{(1-t^2)(1-4s^2-t^2)}}.
\end{eqnarray*}
This follows easily from the fact that, when $a=-1$, 
$$
P=-2(1-t^2)(1-4s^2)^2 \quad \hbox {and} \quad P_{++}P_{-+}P_{+-}P_{--}= \left( (1-4s^2)(1-t^2)\right)^2.
$$
\end{comment}
\end{proof}
\begin{Proposition}
\label{prop:radius-general}
Let $a\ge -1$.  The series $\Ra(a,u,u;1,1)$ that counts square lattice
walks by the length
  (variable $u$) and the number of \NW\ and \ES\  corners ($a$) has radius
  of convergence 
$$
\frac 1 {2+\sqrt{2+2a}}.
$$
The same holds for the series $\Ra_{-,0}(a,u,u;1)$ that counts walks
ending on the $x$-axis.

The series $W(a,u):=\Ra_{0,0}(a,\sqrt u,\sqrt u;1)$ that counts loops
by their half-length and number of \NW\ and \ES\  corners radius of
convergence given by~\eqref{radius}.
\end{Proposition}
\begin{proof}
It follows from Proposition~\ref{prop:general} that
$$
\Ra(a,u,u;1,1)= \frac 1 {1-4u -2{u}^{2}(a-1) }.
$$
This rational series has two poles,
$$
\rho_1=\frac 1 {2+\sqrt{2+2a} }\quad \hbox{and} \quad \rho_2=\frac 1 {2-\sqrt{2+2a}
},
$$
the latter being only defined for $a\not=1$ (recall that we already
assume that $a\ge -1$). For $a \in [-1, 1)$, both
poles are real and positive, with $\rho_1 \le \rho_2$, and thus the
radius is $\rho_1$. For $a\ge 1$, $\rho_1$ is the only positive
singularity, and hence must be the radius by Pringsheim's
theorem~\cite[Thm.~IV.6, p.~240]{flajolet-sedgewick} (we could alternatively
invoke the continuity Lemma~\ref{lem:cont}).

Let us now consider walks ending on the $x$-axis, with \gf\
$$
\Ra_{-,0}(a,u,u;1)= \frac 1 {\sqrt{(1-4u -2{u}^{2}(a-1))(1+2u^2(a-1)) }}.
$$
This series has four singularities, namely $\rho_1$ and $\rho_2$ given
above, as well as
$$
%\rho_1=\frac 1 {2+\sqrt{2+2a} },\quad \rho_2=\frac 1{2-\sqrt{2+2a}},\quad 
\rho_{3,4}= \pm \frac 1 {\sqrt{2-2a}}
$$
which are undefined if  $a=1$. For $a \in [-1, 1)$, all
singularities are real, and $\rho_1$ has minimal modulus, and hence is
the radius. For $a\ge 1$,  the only real positive
singularity is $\rho_1$, which must be the radius by  Pringsheim's
theorem.

Finally, the length generating function of loops satisfies
$$
W(a,u^2)=\Ra_{0,0}(a,u,u)=[x^0y^0] \Ra(a,u,u;x,y).
%=[x^0] \frac 1 {\sqrt{(1-u(x+\bx))^2-4u^2(1-ux(1-a))(1-u\bx(1-a))}}.
$$
%This series is D-finite in $u$, and the Maple package MGfun 
The Mathematica package HolonomicFunctions~\cite{koustchan-creative}
allows one to construct a linear differential equation (DE) in $u$
satisfied by this series, starting 
from a system  of DEs satisfied by the rational series
$\Ra(a,u,u;x,y)$ (one with respect to $u$, one with
respect to $x$, one with respect to $y$).
The DE that we obtain for $W(a,u^2)$ has order two. It translates into a DE of
order 2 for  $W(a,u)$, in which the coefficient of the
second derivative is
\begin{align*}
& u\left(1+2u(a-1) \right)  \left(a+2u (a-1)^2 \right)  \big(1-4u(a+3)+
4{(a-1)}^{2}{u}^{2} \big) \times \\ & \left( a+ u (a-1)(a-3) +2u^2(a-1)^3
\right) .
\end{align*}
The general theory of linear DEs~\cite[p.~519]{flajolet-sedgewick} tells us that the singularities of
$W(a, u)$ are found among the seven roots of this
polynomial, namely, with the above notation:
\begin{align}
& 0, \quad \rho_{1,2}^2=\frac 1 {(2\pm\sqrt{2a+2})^2},\quad
\rho_3^2=\frac 1{2(1-a)}, \quad -\frac a{2(a-1)^2}, \nonumber \\
\label{rho-cand}
& \frac{3-a\pm\sqrt{(9-7a)(1+a)}}{4(a-1)^2}.
\end{align}
It follows from the expressions of $\Ra_{0,0}(a,u,u)$ at $a=1$ and
$a=-1$ given in Proposition~\ref{prop:a+1-R} that the radius of $W(a,
u)$ is $1/16$ at 
$a=1$ and $1/8$ at $a=-1$ (the proof is similar to the proof of the
first part of Proposition~\ref{prop:radiusQ}). Moreover, since $W(a, u)$ is $(a+1)$-positive, the radius is a continuous function of
$a$ on $(-1,+\infty)$, non-increasing on $[-1, +\infty)$. And by
  Pringsheim's theorem, the radius is one of the singularities for
  $a\ge -1$. It then follows from an elementary study of the
  functions~\eqref{rho-cand} that the radius of $W(a,u)$ is 
%$\frac 1 {(2+\sqrt{2a+2})^2}$
$\rho_1^2$
at $a=1$, and
  then, by continuity, for $a\in [-1/2, +\infty)$ 
(since $\rho_1^2$   does not meet any other root in $(-1/2,+\infty)$). 
For $a\in[-1, -1/2]$, we
  have three candidates that would satisfy continuity at $-1/2$ (see Figure~\ref{fig:singularities}), but
  only $-\frac a{2(a-1)^2}$ remains below the value $1/8$, and there are no further intersection points with the
  other two candidates in the
  interval $[-1, -1/2)$.
\end{proof}

\begin{figure}[ht]
\begin{center}
\includegraphics[width=5cm]{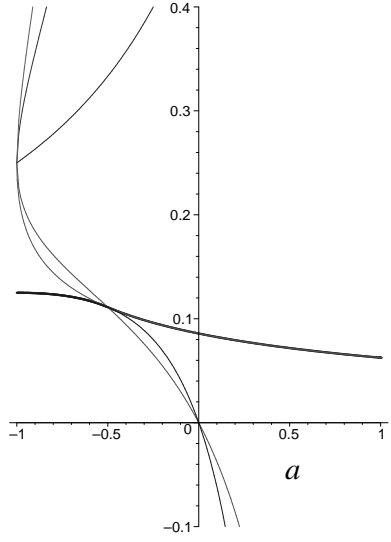}
\caption{The seven candidates for the radius of $W(a,u)$, shown on
  the interval $[-1,1]$. The thick line is the radius.} 
\label{fig:singularities}
\end{center}
\end{figure}

%===================================================
\subsection{Half plane walks}
\label{sec:half-plane}
%===================================================
We obtain similar results for loops confined to
the upper half plane 
$\{(x,y): y\ge 0\}$.
\begin{Proposition}\label{prop:a+1-H}
  The generating function of half plane loops, counted by horizontal steps $(s)$,
  vertical steps ($t$), and \NW\ and \ES\  factors ($a$), is
$$
\Ha_{0,0}(a,s,t)= \sum_{j\ge 0} \frac 1 {j+1}{2j\choose j}
t^{2j} \Ra_{0,0,j}(a,s),
$$
where $ \Ra_{0,0,j}(a,s)$ is given by~\eqref{W00j}.
This series is $(a+1)$-positive.
Moreover,
$$
\Ha_{0,0}(1,s,t)= \sum_{i, j \ge 0} s^{2i} t^{2j}\frac 1{j+1} {2i+2j \choose 2i}
{2i \choose i } {2j\choose j},
$$
while
\beq\label{expr-Hm1}
\Ha_{0,0}(-1,s,t)= \sum_{i, j \ge 0} s^{2i} t^{2j}\frac 1{j+1}{i+j \choose i}
{2i \choose i } {2j\choose j}.
\eeq
For  $a\ge -1$, 
the series $\Ha_{0,0}(a,\sqrt u,\sqrt u)$ that counts half plane loops
 by half-length and corners  has radius of convergence given
by~\eqref{radius}. 
\end{Proposition}

\begin{proof}
 The expression of $\Ha_{0,0}(a,s,t)$ follows from the analogous
 expression~\eqref{g-00} obtained for $\Ra_{0,0}(a,s,t)$ by applying
 Proposition~\ref{prop:ABC}. The same proposition allows us to derive the
 expressions of $\Ha_{0,0}(1,s,t)$ and $\Ha_{0,0}(-1,s,t)$ from their
 counterparts of Proposition~\ref{prop:a+1-R}.

The $(a+1)$-positivity of $\Ra_{0,0,j}(a,s,t)$ implies the
$(a+1)$-positivity of $\Ha_{0,0}(a,s,t)$.

As far as the radius of convergence is concerned, we have for half plane loops
$$
\Ha_{0,0}(a,u,u)= \sum_{j\ge 0} \frac 1 {j+1} {2j\choose j} u^{2j}
\Ra_{0,0,j}(a,u),
$$
while   for general loops,
$$
\Ra_{0,0}(a,u,u)= \sum_{j\ge 0} {2j\choose j} u^{2j}
\Ra_{0,0,j}(a,u).
$$
We have proved in Proposition~\ref{prop:a+1-R} that $\Ra_{0,0,j}(a,u)$ is
$(a+1)$-positive. Hence for $a\ge -1$,
$$
\frac 1{n/2+1} [u^n] \Ra(a,u,u)\le [u^n] \Ha(a,u,u) \le [u^n]
\Ra(a,u,u),
$$
and this proves that $\Ha(a,u,u)$ has the same radius of convergence as $\Ra(a,u,u)$.
\end{proof}

\noindent
\textbf{Remark}. One can also construct the (algebraic) generating functions
$\Ha(a,s,t;x,y)$ (resp.~$\Ha_{-,0}(a,s,t;x)$)  that count
walks confined to the upper half plane (resp.~and ending on the
$x$-axis). When $x=y=1$ and $s=t=u$, the radius of each of these series is
found to be $
\frac 1 {2+\sqrt{2+2a}},
$
as in the unconfined case
(Proposition~\ref{prop:radius-general}). This confirms that the
transition found at $a=-1/2$ is really a property of loops.

%%%%%%%%%%%%%%%%%%%%%%%%%%%%%%%%%%%%%%%%%%%%%%%%%%%%%%%%%%%%%%%%%%%%
\section{Asymptotic Analysis}
\label{sec:asympt}
%%%%%%%%%%%%%%%%%%%%%%%%%%%%%%%%%%%%%%%%%%%%%%%%%%%%%%%%%%%%%%%%%%%%

%======================================
\subsection{Statement of the results}
%======================================
Recall the relationship between the series $Q(a,u)$ and $S\equiv S(t)$
established in Corollary~\ref{cor:QS}: 
\beq\label{QPbis}
     \Q\left(-\Sp ,\frac t {(1+\Sp )^2}\right)= \frac{1+\Sp }{1-\Sp }
\eeq
with $S=1/(1-\Sp)$. Our main theorem below tells us that $S(t)$
reaches its radius of convergence when the pair 
$(-\Sp , t (1+\Sp )^{-2})$ reaches the \emm critical curve,
$\{(a, \rho_Q(a)), a \ge -1\}$, where $\rho_Q(a)$ denotes the radius of the
series $\Q(a, \cdot)$. See Figure~\ref{fig:asympt} for an
illustration. However, this theorem relies on the conjectures
studied in the previous section.

\begin{Theorem}\label{thm:radius}
 Assume that the series $Q(a,u)$ is $(a+1)$-positive, and that
 $Q'_2(a,\rho_Q(a))<\infty$ for $-1/3\le a\le 0$. 
Let $\xc$ be the
 radius of convergence of $S=1/(1-\Sp)$. Then 
 $t/(1+\Sp(t))^2$ increases   on the interval $[0, \xc]$,  and on
 this interval,
%for $t \in [0,\xc]$,
\beq\label{bound}
\frac t{(1+\Sp(t))^2} \le \rho_Q(-\Sp(t)),
\eeq
with equality if and only if  $t=\xc$. Moreover, $\Sp(\xc)\le 1/3$.
\end{Theorem}

\begin{figure}[ht]
\begin{center}
\includegraphics[height=6cm]{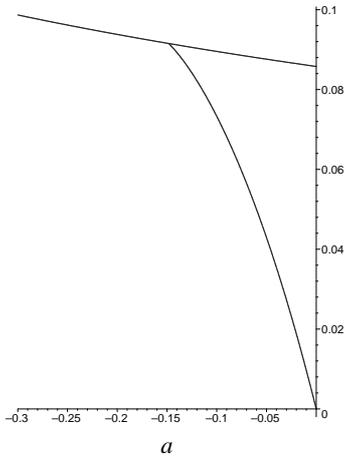}
\caption{The top curve shows the conjectured radius of $Q(a,
  \cdot)$. The bottom curve shows the points $\left(-\Sp(t), t/(1+\Sp(t))^2\right)$ (estimated from the first 70 coefficients of $\Sp$)
as $t$ grows from 0 to $t_c$.} 
\label{fig:asympt}
\end{center}
\end{figure}

This section is devoted to the proof of this theorem. Before we begin
with the proof, let us make the value of $\xc$ more
explicit thanks to the conjectured expression of $\rho_Q(a)$
(Conjecture~\ref{prop:Q0}). 
\begin{Corollary}
Assume that the assumptions of the above theorem hold, as well as (the
first part of) Conjecture~\ref{prop:Q0}.   Then the radius of
convergence of $S$ is 
\beq\label{tc-a}
\xc= \left(1-\frac{\sqrt{2+2a}}2\right)^2,
\eeq
where $a=-\Sp(\xc)$ satisfies
$$
Q\left(a,\frac 1 {(2+\sqrt{2+2a})^2}\right)= \frac{1-a}{1+a}.
$$
\end{Corollary}
%\noindent

\begin{proof}
  Write $a=-\Sp(\xc)$. By Theorem~\ref{thm:radius}, we have $\Sp(\xc)
  \le 1/3 <1/2$, and so by Conjecture~\ref{prop:Q0},
$$
\rho_Q(a)= \frac 1 {(2+\sqrt{2+2a})^2}.
$$
Since~\eqref{bound} is an equality at $t=\xc$, 
$$
\xc= (1-a)^2\rho_Q(a)
$$
and this gives~\eqref{tc-a}. The second identity of the corollary is
obtained by setting $t=\xc$ in~\eqref{QPbis}. 
\end{proof}

Using the first 100 terms of the expansion of $Q(a,u)$ in $u$, 
we estimate $a$ between $-0.15$ and   $-0.148$, which
would give
$$
1/\xc = \limsup s_n^{1/n} \in [8.25, 8.29].
$$
This should be compared with two natural upper bounds on $s_n$:
 the number of operation
sequences  of length $2n$ that output eagerly (that is, have no \NW\
nor \ES\ corner), and the number of  standard operation sequences of length $2n$. According to
Conjecture~\ref{prop:Q0},  the growth constant for  operation
sequences that output eagerly would be $1/\rho_Q(0)= (2+\sqrt 2)^2\simeq 11.6$. Now  the arguments of
Theorem~\ref{thm:eqs} imply that the
generating function 
$\tilde S(t)$ of standard operation sequences satisfies
\beq\label{S-tilde}
\tilde S(t)= 1+ C(1, t\tilde S(t)^2)= 1+ C( t\tilde S(t)^2)
\eeq
if we abbreviate $C(1,v)$ by $C(v)$.  With the same convention,
$$
Q(u)=1+2C( uQ(u)^2).
$$
Recall that $Q(u)\equiv Q(1,u)$ has radius $1/16$. Moreover,
$Q_c:=Q(1/16)=8-64/(3\pi)$. One derives from this that the radius of
$C(v)$ is $Q_c^2/16$, and that at this point $C$ takes the value
$(Q_c-1)/2$. 
Returning to~\eqref{S-tilde}, this implies that $\tilde S$ equals
$(Q_c+1)/2$ at its radius, and that this radius is
$$
\tilde t_c= \frac{Q_c ^2}{4(Q_c+1)^2}.
$$
Taking the reciprocal, this gives the estimate $13.3$ for the growth
constant of standard operation sequences, which is larger than the
growth constant obtained for sequences that output eagerly.

We can also obtain lower bounds on $1/t_c$ directly using the fact
that $(s_n)_{n \geq 0}$ is a super-multiplicative sequence, so
$s_n^{1/n}$ is increasing. At $n = 100$ this gives a bound $7.2 <
t_c$. On the other hand we can do a bit better using $S = 1/(1 -
\Sp)$. Specifically, 
% mbm I just want to check the use of "hence" here. Is it what you
% wanted to write? (sorry for this question...)
on the right hand side we can replace $\Sp$ by a
polynomial truncation of its Taylor series to obtain a power series
dominated term by term by $S$ whose radius of convergence therefore is
not smaller than $t_c$. This approximation gives $7.38 < 1/t_c$, using
the truncation of $\Sp$ of degree 100.  
% 

%============================================================
\subsection{Relating  the singularities of  $S$ and $C$}
%============================================================
We begin with a simple lemma.
\begin{Lemma}\label{lem:S}
The series $\Sp(t)$ and $S(t)=1/(1-\Sp(t))$ have the same radius of
convergence $\xc$. Moreover,  $\Sp(\xc) <1$, so that  $S(\xc)< \infty$.
% In fact, we even have $T(\rho)<1/2$.
\end{Lemma}
\begin{proof}
  Let $s_n$ (resp.~$\coSp_n$) denote the coefficient of $t^n$ in $S(t)$
  (resp.~$\Sp(t)$). Then 
$
\coSp_n \le s_n
$, since $\coSp_n$ counts primitive achievable permutations (of size $n$), while $s_n$ counts
  all achievable \ps. Recall that
  $s_n$ also counts  canonical operation
  sequences.  If $w$ is a canonical operation
  sequence (seen as a word on $\{\NN,\SS,\EE,\WW\}$), then $\EE w\WW$ is a
  primitive canonical operation  sequence. This shows that $s_n \le \coSp_{n+1}$.

It follows from these inequalities that $S$ and $\Sp$ have the same
radius of convergence $\xc$. The identity $S=1/(1-\Sp)$ then gives
$\Sp(\xc) \le 1$ (otherwise the radius of $S$ would be smaller than
that of $\Sp$). In particular, the series $\Sp(t)$ is convergent at
$t=\xc$. The inequality $s_n \le \coSp_{n+1}$ then implies that also
$S(t)$ is convergent at $t=\xc$. This implies in turn that $\Sp(\xc)<1$.
%
%That $T(\xc)<1/2$ still needs a proof.
\end{proof}

Our next lemma exploits the connection between the series $C(b,v)$ and
$S(t)$ established in Theorem~\ref{thm:eqs}, which can be written as:
\begin{equation}
\label{FE:SCbis}
S(t) = 1 + C \left( \Sp, tS^2 \right).
\end{equation}

\begin{Lemma}\label{lem:insideC}
    Let $\xc$ be the radius of convergence of $S=1/(1-\Sp)$, and
    $\rho_C(b)$ the radius of convergence of $C(b, \cdot)$. Then 
    $S(t)$ increases   on the interval $[0, \xc)$,  and on this interval,
$$
tS(t)^2 < \rho_C(\Sp(t)).
$$
\end{Lemma}
\begin{proof}
  That $S(t)$ increases is obvious since the series $S$ has
  non-negative coefficients. We now argue \emm ad absurdum,. Assume that
  there exists $t_1<\xc$ such   that $
t_1S(t_1)^2 \ge \rho_C(\Sp(t_1)).
$
Let $t_2\in (t_1, \xc)$.  Since $S(t)$ increases strictly with $t$ while
$\rho_C(\Sp(t))$ decreases weakly,  $t_2S(t_2)^2> \rho_C(\Sp(t_2)).
$  
Let us write $C(b,v)= \sum_{k \ge 0, m\ge 1} c_{k,m} b^k v^m$. The
identity~\eqref{FE:SCbis}  
%relating the power series $C$ and $S$ 
gives, for $n\ge 1$,
\beq\label{SC-co}
\begin{array}{l} \displaystyle{s_n :=[t^n] S(t)= \sum_{k \ge 0, m\ge 1} c_{k,m} a_{n,k,m}} \quad
\hbox{where} \\
\displaystyle{a_{n,k,m}:=[t^n] \left(
  \Sp(t)^k t^m S(t)^{2m}\right) \ge 0.}
\end{array}
\eeq
Let us now evaluate the series 
$C(b,v)$ at $b=\Sp(t_2)$ and $v=t_2S(t_2)^2$. Since $t_2S(t_2)^2>
\rho_C(\Sp(t_2))$
%$(b,v)$ lies out of the domain of convergence of $S$, 
and $S$ has non-negative coefficients, this series should be
infinite. However,
\begin{eqnarray*}
 C(b,v) \ = \  \sum_{k \ge 0, m\ge 1} c_{k,m} b^ k v^m&=& \sum_{k \ge 0, m\ge 1}
c_{k,m}\Sp(t_2)^k t_2^m S(t_2)^{2m}\\
&= &\sum_{k \ge 0, m\ge 1} c_{k,m}\sum_ {n\ge 1} {t_2^n} a_{n,k,m}
\\
&= &\sum_ {n\ge 1} {t_2^n} \sum_{k \ge 0, m\ge 1} c_{k,m} a_{n,k,m}
\\
&=&\sum_ {n\ge 1} {t_2^n}s_n \hskip 15mm  \hbox{by}~\eqref{SC-co}\\
&=& S(t_2) - 1 <\infty \hskip 12mm  \hbox{since } t_2 < \xc.
\end{eqnarray*}
In the third line, we have used the fact that all terms in the sum are
non-negative, so that the value of the series is unchanged if we perform
any rearrangement of terms. 

We have thus obtained a contradiction, and the lemma is proved.
\end{proof}

The next lemma deals with the series $C(b,v)$ and its radius
$\rho_C(b)$. The proof is given in Appendix~\ref{APP:CSmall}. It is purely
combinatorial and in particular,  does not use the
equations of Section~\ref{sec:enumeration}. 

\begin{Lemma} \label{lem:C<1/2}
Let $b >0$. Then $\rho_C(b) \le 1/4$ and for
   $v \in  [0, \rho_C(b))$, 
$$
v  < C(b,v) <  \frac{1}{2}.
$$
The series $A(b,\cdot)$ and $U(b,\cdot)$ defined
  by~\eqref{UBAT}  have radius of  convergence at least $\rho_C(b)$.\\
 The series $C(0, \cdot)$, $A(0, \cdot)$ and $U(0, \cdot)$ have
 respectively radius $+\infty$, $+\infty$ and $1/2$.
\end{Lemma}

\begin{Corollary}\label{cor:Sp13}
  For $t\in [0, \xc]$ one has
$$
S(t) \le \frac 3 2 \quad \hbox{and} \quad \Sp(t) \le \frac 1 3.
$$
\end{Corollary}
\begin{proof}
The identities are obvious if $t=0$, so let us assume $t>0$.
Then $\Sp(t)>0$, and by Lemma~\ref{lem:insideC}, the pair $( \Sp(t), tS(t)^2)$ lies in
  the domain of convergence of $C(b,v)$ for $t\in [0, \xc)$. Hence~\eqref{FE:SCbis}
  holds in this interval, and implies that $S(t)\le 3/2$ by
  Lemma~\ref{lem:C<1/2}.  Since $S=1/(1-\Sp)$, this means that
  $\Sp(t)\le 1/3$ in this interval. These inequalities hold at $t=\xc$
  as well by continuity. 
\end{proof}

%============================================================
\subsection{Relating the singularities of  $S$ and $Q$}
%============================================================
We first establish a weak form of Theorem~\ref{thm:radius}.

\begin{Lemma}\label{lem:tight}
Assume that the series $Q(a,u)$ is $(a+1)$-positive.   There exists
$t_1\in [0, \xc]$ such that  
$$
\frac {t_1}{(1+\Sp(t_1))^2} = \rho_Q(-\Sp(t_1)).
$$
Moreover for any such $t_1$, the function $t(1+\Sp(t))^{-2}$ is
increasing on $[0, t_1]$. 
\end{Lemma}
\begin{proof}
Recall that $\Sp(t)<1/3  $
  for $t\in [0, \xc]$  (Corollary~\ref{cor:Sp13}), and   assume
  that the first part of the lemma is wrong. By continuity of $\Sp$
  and $\rho_Q$, this 
  means that~\eqref{bound} holds strictly on $[0, \xc]$. Then  for $t\in [0, \xc]$, 
  the pair $(-\Sp(t), t (1+\Sp(t))^{-2})$ lies in the (open) domain of convergence of
  $Q$, and by Corollary~\ref{cor:QS},
\beq\label{Qsp}
Q\left( -\Sp(t), \frac t {(1+\Sp(t))^2}\right)= \frac{1+\Sp(t)}{1-\Sp(t)}.
\eeq
This holds in particular at $t=\xc$. We will now use the implicit
function theorem to define an analytic continuation of $\Sp$ at
$t_c$. Consider the equation 
$$
Q\left( -\Sb(t), \frac t {(1+\Sb(t))^2}\right)= \frac{1+\Sb(t)}{1-\Sb(t)}
$$
as the implicit definition of a function $\Sb(t)$. The implicit
function theorem guarantees the existence of a (unique) analytic
solution $\Sb(t)$ defined in a neighbourhood of $\xc$ and satisfying
$\Sb( \xc)= \Sp(\xc)$, provided
$$
-Q'_1
%\left( -\Sp(\xc ), \frac \xc  {(1+\Sp(\xc ))^2}\right)
 - \frac {2\xc}  {(1+\Sp(\xc ))^3} Q'_2
% \left( -\Sp(\xc ), \frac \xc  {(1+\Sp(\xc ))^2}\right)
\not = \frac 2 {(1-\Sp(\xc ))^2},
$$
where $Q'_1$ and $Q'_2$ denote the derivatives of $Q$ taken at 
\[
\left(-\Sp(\xc ),  \xc  {(1+\Sp(\xc ))^{-2}}\right). 
\]
But the
$(a+1)$-positivity of $\Q$,  
together with the fact that $\Sp(\xc)<1$, implies that the left-hand
side is negative, while the right-hand side is positive. So the
implicit function theorem applies. By~\eqref{Qsp}, the function $\Sb$
must coincide with $\Sp$ on an interval of the form $(\xc-\vareps,
\xc)$, for some $\eps>0$. It thus constitutes an analytic continuation of $\Sp$ at $\xc$, 
which is impossible by Pringsheim's theorem (see~\cite[Thm.~IV.6,
p.~240]{flajolet-sedgewick}).  We have thus reached a contradiction,
which proves 
 the first part of the lemma.

\medskip
Now~\eqref{Qsp} holds for $t\in [0, t_1]$ by analytic continuation
and continuity at $t_1$. The
right-hand side increases with $t$, and thus the left-hand side must
also increase. However, due to the $(a+1)$-positivity of $Q(a,u)$, it
reads
$$
\sum_{k,n\ge 0} q_{k,n} (1-\Sp(t))^k
\left(\frac{t}{(1+\Sp(t))}\right)^n,
$$
with $q_{n,k}\ge 0$, and if ${t}{(1+\Sp(t))^{-2}}$ would decrease,
even locally (or weakly), so would the whole left-hand side (because
$(1-\Sp(t))$  decreases). Hence ${t}{(1+\Sp(t))^{-2}}$ increases on
$[0, t_1]$.
 \end{proof}

\begin{figure}[ht]
\begin{center}
\input{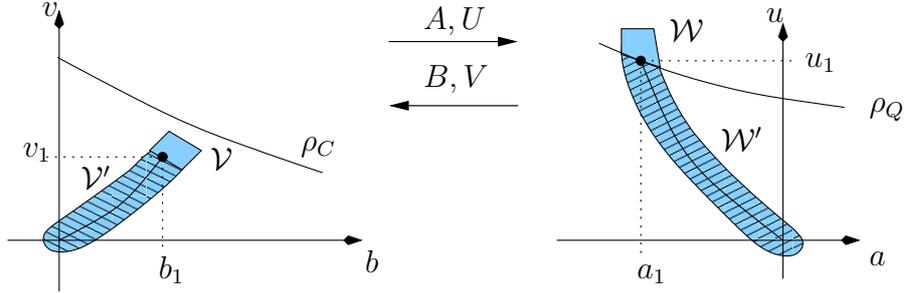}
\caption{Illustration for the proof of Theorem~\ref{thm:radius}. Left:
the $(b,v)$-plane of the series $C(b,v)$. Right: the $(a,u)$-plane of
the series $Q(a,u)$.} 
\label{fig:CQplanes}
\end{center}

\end{figure}
We are now ready for the
 \begin{proof}[Proof of Theorem~\ref{thm:radius}.]
The bound~\eqref{bound} holds strictly at $t=0$, but by
Lemma~\ref{lem:tight}, it cannot be strict on $[0, 
t_c]$. Let $t_1$ be the smallest value of $[0, t_c]$ where the
equality holds. We have to prove that $t_1=t_c$. We argue \emm ad
absurdum,.   The argument is illustrated by 
Figure~\ref{fig:CQplanes}.    Let us denote 
\beq\label{a1}
\begin{array}{lllllllll}
a_1&=& -\Sp(t_1), & \hskip 10mm u_1&=& \displaystyle \frac{t_1}{(1+\Sp(t_1))^2},\\
b_1&=&\Sp(t_1), &\hskip 10mm  v_1&=& \displaystyle  \frac{t_1}{(1-\Sp(t_1))^2}&=& t_1 S(t_1)^2.
\end{array}
\eeq
As in the proof of the previous lemma, our objective is to obtain a
contradiction by constructing an analytic continuation of the map
$u\mapsto Q(a_1,u)$ at $u=u_1$. However, it will take a bit of work
before we can establish our starting point, namely  that~\eqref{Qa1}
holds on an interval $[u_1-\vareps, u_1]$.

By Lemma~\ref{lem:insideC}, the closed curve $\overline{\cC}_C:=\{(\Sp(t), t
    S(t)^2),t\in [0, t_1]\}$ lies in the region
    $\cD_C=\{(b,v)     \in \cs^2, |v| <
    \rho_C(|b|)\}$.
 We adopt the convention
\beq\label{convention}
\rho_C(0)=\lim_{b\rightarrow 0^+} \rho_C(b),
\eeq
which makes  $\cD_C$ open and connected (recall that $C(0,v)=v$, so
that the radius of $C(0, \cdot)$ is 
infinite; 
  Lemma~\ref{lem:C<1/2} implies that the above value of
$\rho_C(0)$ is less than $1/4$.
Note that $C$ is analytic in $\cD_C$, as well as $A$ and $U$ 
    (by Lemma~\ref{lem:C<1/2}). Hence~\eqref{FE:SCbis} holds for $t\in
    [0, t_1]$, and the definition of $A$ and $U$ in terms
    of $C$ (i.e.~equation~\eqref{UBAT}) thus gives
\beq\label{Acurve}
A(\Sp, tS^2)=-\Sp \quad \hbox{and}\quad U(\Sp, tS^2) = \frac
t{(1+\Sp)^2}.
\eeq
    Let  $\cV $ be an open neighbourhood of $\overline{\cC}_C$ 
    contained in $\cD_C$. By the open mapping 
    theorem in two variables~\cite[Thm.~6.3]{kaup}, 
the image  by $(A,U)$ of $\cV $ is
    a  neighbourhood $\cW $ (in $\cs^2$) of
    $\overline{\cC}_Q:= (A,U)\left(\overline{\cC}_C \right) =
    \{(-\Sp(t), t     (1+\Sp(t))^{-2}), 
    t\in [0, t_1]\}$ (see Figure~\ref{fig:CQplanes}). 
  Let  $\cD_Q=\{(a,u)    \in \cs^2, |u| <
    \rho_{\Q}(|a+1|-1)\}$ (with the same convention as in~\eqref{convention} for
 defining   $\rho_Q(-1)$). Then by continuity of $\rho_Q$ (Proposition~\ref{prop:radiusQ}), 
$\cD_Q$ is open and connected, and $Q$ is analytic in  $\cD_Q$. By
definition of $t_1$, the domain $\cD_Q$    contains  ${\cC_Q}:=\{(-\Sp(t), t
    (1+\Sp(t))^{-2}), 
    t\in [0, t_1)\}$.  Let $\cW '$
    be the (necessarily open) connected component of  $ \cW \cap
    \cD_Q$ containing ${\cC_Q}$, and let
$$
\cV '= \{(b,v) \in \cV  : (A(b,v), U(b,v)) \in \cW '\}.
$$
Then
% $\cW '$ is a neighbourhood of ${\cC_Q}:=\{(-\Sp(t), t   (1+\Sp(t))^{-2}),     t\in [0, t_1)\}$, and similarly
 $\cV '$ is a connected open neighbourhood of ${\cC_C}:=\{(\Sp(t), t
    S(t)^2),t\in [0, t_1)\}$. By analytic continuation
    of~\eqref{CfromQ},     we have,     for $(b,v) \in \cV '$,
\beq\label{6bis}
Q(A(b,v), U(b,v))= 1+ 2C(b,v),
\eeq
and in particular $Q(A(b,v), U(b,v))\not = 0$ since $|C(b,v)|< 1/2$
in $\cD_C$ (see Lemma~\ref{lem:C<1/2}). Hence $1/Q(A,U)$ has no pole
in $\cV '$, the series $B(A,U)$ is analytic in $\cV'$, and by analytic
continuation of~\eqref{BAT},  
    we have,     for $(b,v) \in \cV '$,
\beq\label{BAVA}
B( A(b,v),U(b,v))=b \quad \hbox{and} \quad V(A(b,v),U(b,v))=v.
\eeq
Let $(a,u)\in \cW '$.  By definition of $\cW$ and $\cV '$, there
exists $(b,v) \in \cV '$ such that 
$A(b,v)=a$ and $U(b,v)=u$. The above identities show that $b$ and $v$
are unique, and given by
$$
b=B(a,u) \quad  \hbox{and } \quad v=V(a,u).
$$
In particular, the identity~\eqref{6bis} reads, for $(a,u)\in \cW '$,
$$
Q(a,u)=1+ 2C \left( 1 - \frac{1-a}{Q(a, u)}, u\Q(a, u)^2 \right).
$$
Recall that $\cW '$ is the connected component of $\cW  \cap \cD_Q$ containing
$\cC_Q$, and that $\cW $ contains a ball centered at   the point $(a_1, u_1)$. This implies that $\cW '$
contains  a segment $\{(a_1, u): u \in [u_1 -\eps,
u_1)\}$ with $\eps>0$. Hence the identity
\beq\label{Qa1}
Q(a_1,u)=1+ 2C \left( 1 - \frac{1-a_1}{Q(a_1, u)}, u\Q(a_1, u)^2 \right)
\eeq
holds in this segment, and by continuity at $u_1$ as well. Taking the
limit $(b,v)\rightarrow (b_1, v_1)$ in~\eqref{BAVA} shows,
in combination with~\eqref{a1} and~\eqref{Acurve}, that 
$$
 1 - \frac{1-a_1}{Q(a_1, u_1)}=b_1, \quad \hbox{and} \quad  u_1\Q(a_1,
 u_1)^2=v_1.
$$
Recall that $C$ is analytic in the neighborhood of $(b_1, v_1)$. We can now
mimic the implicit function argument used in the proof of
Lemma~\ref{lem:tight}. 
 Consider the equation 
\beq\label{Qc-def}
\Qc(u)=1+ 2C \left( 1 - \frac{1-a_1}{\Qc( u)}, u\Qc( u)^2 \right)
\eeq
as the implicit definition of a function $u\mapsto \Qc(u)$. The implicit
function theorem guarantees the existence of a (unique) analytic
solution  defined in a neighbourhood of $u_1$ and satisfying
$\Qc(u_1)= Q(a_1,u_1)$, provided that
\beq\label{not1}
1\not = 2\frac{1-a_1}{\Q(a_1, u_1)^2}C'_1+ 4u_1 \Q(a_1,u_1) C'_2,
\eeq
where the derivatives $C'_1$ and $C'_2$ are taken at the point $(b_1,v_1)$.
%$\left(1 - \frac{1-a_1}{\Q(a_1, u_1)}, u_1\Q( a_1,u_1)^2\right)$.
%that is, at $(b_1,v_1)$ (by definition of $B$ and $\cV $ and~\eqref{BAVA} applied in the limit$(b,v)\rightarrow (b_1, v_1)$). 
By differentiating~\eqref{Qa1} with respect to $u$, we obtain, for
$u\in [u_1-\eps, u_1)$,
$$
Q'_2(a_1,u)= Q'_2(a_1,u)\left( 2\frac{1-a_1}{\Q(a_1, u)^2}C'_1+ 4u
  \Q(a_1,u) C'_2\right) +2 \Q(a_1, u)^2C'_2,
$$
where the derivatives are evaluated at $\left(1 - \frac{1-a_1}{\Q(a_1,
    u)}, u\Q( a_1,u)^2\right)$. Recall that by definition of $t_1$,
the point $(a_1, u_1)$ lies on the critical curve of $Q$. Since we
have assumed that $\Q_2'(a,u)$ is finite on this curve,
There exists in a neighbourhood of $u_1$ a
(unique) analytic function $\Qc(u)$ satisfying~\eqref{Qc-def} and
$\Qc(u_1)= Q(a_1, u_1)$. By~\eqref{Qa1}, it coincides with $Q(a_1,u)$
on the segment $[u_1-\vareps, u_1)$, and thus constitutes an analytic
continuation of $u\mapsto Q(a_1, u)$ at $u_1=\rho_Q(a_1)$. This
contradicts Pringsheim's theorem, and we have thus proved that $t_1=t_c$.

It now follows from Lemma~\ref{lem:tight} that  $t(1+\Sp(t))^{-2}$ is
increasing on $[0,\xc]$. Finally, the bound on $\Sp(\xc)$ comes from Corollary~\ref{cor:Sp13}.
 \end{proof}

%%%%%%%%%%%%%%%%%%%%%%%%%%%%%%%%%%%%%%%%%%%%%%%%%%%%%%%%%%%%%%%%%%
\section{Some questions and observations}
\label{sec:final}
%%%%%%%%%%%%%%%%%%%%%%%%%%%%%%%%%%%%%%%%%%%%%%%%%%%%%%%%%%%%%%%%%%
The work we have presented
opens up some obvious related or more general questions.

\subsection{Some questions raised directly by our work}

Of course,  the conjectures of
Section~\ref{sec:conj} remain open. For the readers who would be
interested in exploring them, we discuss some possible improvements of
these conjectures in Section~\ref{subsec:more a+1}. 

Also, we do not know anything about the nature of the
series $S(t)$: is it D-finite, is it differentially algebraic? Nor do
we know  the nature of $Q(a,u)$ for a generic value 
of $a$, nor even for $a=0$.
 The right plot of Figure~\ref{fig:radius}
suggests that the  exponent in the asymptotic behaviour of $q_n(a,u)$
varies continuously  
with $a$ (and is not constant), which would rule out D-finiteness for
a generic value of $a$. 
For comparison, for unconfined loops we
predict from the differential equation that the exponent is $-1$ for
$a\ge -1$, except at $a=-1/2$ where it is  $-3/4$ (this could almost
certainly be made into a rigorous proof). For loops confined to the
upper half plane, we find an exponent $-2$ except at $a=-1/2$ where it
should be $-7/4$.

Finally, it would be interesting to obtain an  asymptotic estimate of
$s_n$, not only its exponential growth constant. We have submitted the
first 70 values of $s_n$ to Tony Guttmann who predicts, using 
differential approximants, that $s_n \sim \kappa\, \xc^{-n} n^{-\gamma}$,
for some positive constant $\kappa$, where 
$\gamma\simeq -2.48$ might be $-5/2$ if we expect it to be rational,
and $t_c\simeq 0.12075$, so $1/t_c \simeq 8.28$.

\subsection{Other rearranging devices}
All of the questions we have asked about stacks in parallel can equally well be asked about other
devices whose only purpose is to permute data. Specifically we could consider double ended queues (deques)
and multiple stacks in parallel.

The action of a deque was considered
by Knuth~\cite[Sec.~2.2.1]{Knuth68}. 
A deque behaves very much like two stacks 
in parallel, treating the inputs at either end as  
corresponding  to inputs to two stacks. The difference is that the
bottoms of the stacks are effectively connected meaning that  
an element can be input at one end and output from the other. The arch system diagrams extend naturally to this context 
viewing the whole picture as a cylinder (by connecting the upper and lower edges), so we are also allowed ``arches'' 
that loop around --- starting above the line and finishing below it (or vice versa). Of course the non-crossing criterion must 
still be satisfied. In this case there are further sources of
non-uniqueness and one needs to develop a new notion of canonical sequences.

Likewise, one could consider a system of $m$ stacks in parallel
for any $m \geq 2$. Operation sequences now correspond to loops in
$\ns^m$, and the arches of our arch systems are now coloured with $m$
colours instead of 2. Section~\ref{sec:canonical} extends without any
difficulty,  provided we define for each connected arch system a \emm standard,
colouring. The main difficulty comes later, when one relates loops in
$\ns^m$ to connected arch systems: one has to determine in how many
ways a standard connected arch system can be re-coloured (when $m=2$,
this is the factor 2 in Eq.~\eqref{FE:QC}), and this question requires
further investigation.

%=====================================================================
\subsection{More on $(a+1)$-positivity of loops}
\label{subsec:more a+1}
%=====================================================================
In our attempts for proving Conjecture~\ref{conj:a+1} (the generating function of quarter
plane loops is $(a+1)$-positive), we have tried to see if stronger
properties hold. We believe that the following observations  may be useful for the
readers who would be interested in exploring this conjecture.

\subsubsection{Some properties that may hold}
We begin with a strong property dealing with the values found at
$a=-1$. If true, it would give a new proof
of~\eqref{expr-Qm1},~\eqref{expr-Rm1} and~\eqref{expr-Hm1}.  Below, we
call a \emm bilateral Dyck path, any one-dimensional walk starting and
ending at $0$.
\begin{enumerate}
\item [{$(P_1)$}] Let $w$ (resp.~$v$)  be a  bilateral Dyck path of half-length $i$ (resp.~$j$)  
on the alphabet $\{\EE,\WW\}$ (resp.~$\{\NN,\SS\}$). 
Then the polynomial that counts walks of the shuffle
  class of $vw$ according to the number of \NW\ and \ES\  corners takes
  the value ${i+j \choose i}$ at $a=-1$. 
\end{enumerate}
% \noindent
By Lemma~\ref{cor:equidistributed}, replacing the pair (\NW, \ES) by
(\NW, \SE), (\WN, \ES) or
(\WN, \SE) does not change the validity of the statement.
This property, observed by Julien Courtiel, has been checked for $i,j
\le 5$ for Dyck paths, and for $i,j\le 4$ for bilateral Dyck paths.

% \medskip
Our second property deals with $(a+1)$-positivity. We have proved in
this paper that for any bilateral Dyck path $v$  on the alphabet
  $\{\NN,\SS\}$,  the 
  generating function of loops that project vertically on $v$, counted by the length and the number
  of \NW\ and \ES\  corners, is $(a+1)$-positive. In
  fact, this series only depends on the length of $v$
  (see Propositions~\ref{prop:ABC} and~\ref{prop:a+1-R}). A similar
  statement might be true for quarter plane loops.
\begin{enumerate}
\item [{$(P_2)$}]  Let  $v$  be a  Dyck path of half-length $j$  on the alphabet
  $\{\NN,\SS\}$. Then the 
  generating function of quarter plane loops that project vertically on $v$
%, and remain in the  half plane $\{(x,y): x\ge 0\}$, counted by the
%length and the number  of \NW and \ES\  corners, 
is $(a+1)$-positive (but does not depend on $j$ only).
\end{enumerate}
% \noindent 
By Lemma~\ref{cor:equidistributed}, replacing the pair
(\NW, \ES) by (\NW, \SE), (\WN, \ES) or
(\WN, \SE) does not change the validity of the statement.
This property has been checked for $j\le 5$ and loops of half length
at most 10.

%===========================================================
\subsubsection{Some properties that do not hold}
Our first observation is that $(a+1)$-positivity really appears as a
property of loops.
\begin{enumerate}
\item  [{$(N_1)$}] There is no $(a+1)$-positivity property for walks ending
  at a prescribed endpoint $(i,j)$, whether confined to the quarter plane, to
  the upper half plane or not confined at all.
\end{enumerate}
% \noindent
\textbf{Examples}. For unconfined walks of length 3 ending at $(-1,2)$,
we obtain the polynomial $2a+1$. Since these walks are confined to the
upper half plane, this also provides an example in this case. Finally,
for quarter plane walks of length 7 ending at $(5,0)$, we obtain the
polynomial $15a+12$. 

% \medskip
\begin{enumerate}
\item  [{$(N_2)$}] There is no  $(a+1)$-positivity property inside a shuffle
  class, even in the quarter plane. 
\end{enumerate}
% \noindent 
\textbf{Example}. For the shuffle class of ({\sf EWEWEW},{\sf
  NNNSSS}), we find the polynomial $62a^3+292a^2+390a+180$, which is not $(a+1)$-positive.

However, the value at $a=-1$ is conjectured to be very simple (and
positive), see Property $(P_1)$ above.

% \medskip
We finally examine a natural extension of $(P_2)$ to bilateral Dyck paths.
\begin{enumerate}
\item  [{$(N_3)$}] There is no  $(a+1)$-positivity property for loops of the
  half plane  $\{(x,y): x\ge 0\}$ that project on a fixed bilateral
  Dyck path $v$.
\end{enumerate}

\noindent 
\textbf{Example}. For $v={\sf SSNN}$, the series reads
$$
u^4+ (4a^2+6a+5)u^6+O(u^8),
$$
and the second coefficient is not $(a+1)$-positive.

\bigskip
\noindent
{\bf{Acknowledgements.}} We are  indebted to  Cyril Banderier, Olivier
Bernardi, Alin Bostan, Julien
Courtiel, Tony Guttmann, Pierre
Lairez, Kilian Raschel  for helpful and interesting discussions. 
MA thanks LaBRI for its hospitality  during visits in 2008 and 2012. 

% mbm I have changed the format of the Denton reference, because the
% arxiv number did not appear with the previous format.

\appendix

%

%%%%%%%%%%%%%%%%%%%%%%%%%%%%%%%%%%%%%%%%%%%%
\section{The series $\boldsymbol{Q(-1,s,u)}$}
\label{app:m1}
%%%%%%%%%%%%%%%%%%%%%%%%%%%%%%%%%%%%%%%%%%%%%%%%%%%%%%%%%%%%

We now  prove the second part of Proposition~\ref{prop:Qm1}, dealing
with the case $a=-1$. We start from the
functional equation~\eqref{eq:Qa-ref}.
As a warm up, let us give another
proof of the case $a=1$, based on that equation.
 Our approach  is taken from~\cite{mbm-mishna}. When
 $a=1$, Equation~\eqref{eq:Qa-ref} reads: 
\beq\label{eqQ1}
K(x,y)xy\Qa(x,y)= xy-ux \Qa(x,0)-usy \Qa(0,y),
\eeq
with $K(x,y)\equiv K(s,u;x,y)= 1-u(sx+s\bx +y+\by)$. Observe that the
variables $x$ and $y$ are decoupled in the unknown series occurring in
right-hand side. Moreover, $K(x,y)$
is left unchanged by the two following involutions:
$$
(x,y)\mapsto (\bx, y) \quad \hbox{and} \quad (x,y)\mapsto (x, \by).
$$
Each involution fixes one coordinate of the pair $(x,y)$: this will play an
important role in the solution.
Together, these involutions generate a group of order 4, and the orbit of $(x,y)$
is $\{(x,y),  (\bx, y), (\bx, \by), (x, \by)\}$. Let us form the
alternating sum of~\eqref{eqQ1} over this orbit. Because of the
$x/y$-decoupling, all unknown 
series on the right-hand side disappear, leaving
\begin{align*}
K(x,y) \left( xy\Qa(x,y)- \bx y\Qa(\bx,y)+  \bx \by\Qa(\bx,\by)- x\by\Qa(x,\by)\right)
& = \\
xy-\bx y +\bx\by -x \by 
= (x-\bx)(y-\by).
\end{align*}
Equivalently,
$$
xy\Qa(x,y)- \bx y\Qa(\bx,y)+  \bx \by\Qa(\bx,\by)- x\by\Qa(x,\by)
= \frac{(x-\bx)(y-\by)}{ 1-u(sx+s\bx +y+\by)}.
$$
To conclude, we observe that, on the left-hand side, the series
$xy\Qa(x,y)$ consists of monomials in which the exponents of $x$ and
$y$ are positive.  In the other three series, either the exponent of
$x$, or the exponent of $y$ (or both) is negative. This tells us that
$xy\Qa(x,y)$ is the \emm positive part in $x$ and $y$, of the rational
series occurring on the right-hand side. In particular, extracting the
coefficient of $x^1y^1$ in the above equation gives
\begin{eqnarray*}
  \Qa(0,0)\equiv\Qa(1,s,u;0,0)&=& [xy] \frac{(x-\bx)(y-\by)}{
    1-u(sx+s\bx +y+\by)}
\\
&=& \sum_{n\ge 0} u^n [xy]\big(  (x-\bx)(y-\by)(sx+s\bx
  +y+\by)^n\big),
\end{eqnarray*}
which yields
$$
 \Qa(1,s,u;0,0)= \sum_{n\ge 0} u^{2n} \sum_{i=0}^n s^{2i} {2n\choose
   2i} C_i C_{n-i}
$$
after an elementary calculation. This is equivalent to the
expression~\eqref{Q1-ref} of $Q(1,s,u)=\Qa(1, \sqrt s, \sqrt u;0,0)$.

% \medskip 
Let us now move to the solution of~\eqref{eq:Qa-ref} in the
case $a=-1$. The equation reads 
\beq\label{eq:xy}
K(x,y)  xy\Qa(x,y)=
xy-ux (1-2usx)\Qa(x,0)-suy (1-2uy)\Qa(0,y) ,
\eeq
where now
\beq\label{Km1}
K(x,y)= 1-u(sx+s\bx+y+\by)+2u^2s(x\by+y\bx).
\eeq
The involutions that leave $K(x,y)$ unchanged and fix an element of
the pair $(x,y)$ are now
$$
(x,y) \mapsto \left(\bx\, \frac{1-2uy}{1-2u\by},y\right)
\quad \hbox{and} \quad 
(x,y) \mapsto \left(x, \by\, \frac{1-2usx}{1-2us\bx}\right).
$$
However, they generate an infinite group, which prevents us from
applying the above strategy. But a finite group is still hiding in
this equation. Let us  introduce new variables $X$ and $Y$, with
$$
x=2us+X \quad \hbox{and} \quad y= 2u+Y.
$$
The functional equation~\eqref{eq:xy} now reads
\begin{multline}
  \label{eqQm1}
\tilde K(X,Y)  XY\Qta(X,Y)= \\ (2us+X)(2u+Y)
-u(2us+X) (1-2us(2us+X))\Qa(2us+X,0)\\
-us(2u+Y) (1-2u(2u+Y))\Qa(0,2u+Y) ,
\end{multline}
with $\Qta(X,Y)=\Qa(2us+X,2u+Y)$ and
$$
\tilde K(X,Y)= 
\frac{xyK(x,y)}{XY}=1- 4u^2(1+s^2) -suX-uY - su \alpha  \bX - u\beta \bY,
$$
where $\bX=1/X$, $\bY=1/Y$, and $\alpha=(4u^2s^2-1)$ and $\beta=(4u^2-1)$ are independent of $X$
and $Y$. This Laurent polynomial is now invariant by the (simpler) involutions
$$
(X,Y) \mapsto ( \alpha \bX , Y)
\quad \hbox{and} \quad 
(X,Y) \mapsto (X, \beta \bY).
$$
 These involutions generate again a group of order four\footnote{In
   terms of the original variables $x$ and $y$, these involutions are
\[
\Phi: (x,y) \mapsto \left( \frac{2su-\bx}{1-2su\bx},y\right) \: \mbox{and} \:
   \Psi: (x,y)
   \mapsto \left( x,\frac{2u-\by}{1-2u\by}\right).
 \] 
 They do not leave
   $K(x,y)$ invariant, but transform it simply as follows:
$$
K(\Phi(x,y))= \frac {1-4s^2u^2}{(1-2sux)(1-2su\bx)}\, K(x,y)
\: \hbox{and} \: K(\Psi(x,y))=  \frac {1-4u^2}{(1-2uy)(1-2u\by)}\, K(x,y).
$$}, and the orbit of $(X,Y)$ is now
$$
\{ (X,Y),  \left(\alpha \bX, Y\right), 
 \left( \alpha \bX ,  \beta \bY\right), \left( X, \beta \bY\right)\}.
$$
We form the alternating sum of~\eqref{eqQm1} over this orbit:
\begin{multline*}
  \tilde K(X,Y) \times \\
  \left( XY \Qta(X,Y)- \alpha \bX Y \Qta(\alpha \bX ,Y)
+\alpha \beta \bX \bY \Qta(\alpha \bX ,\beta \bY) - \beta X \bY
\Qta(X,\beta \bY)\right) \\
%XY -\alpha \bX Y +\alpha \beta \bX \bY - \beta X \bY
=(X-\alpha \bX)(Y-\beta \bY).
\end{multline*}
Returning to the original variables $x$ and $y$, this gives, after
dividing by $\tilde K(X,Y) $,
\begin{multline*}
(x-2su)(y-2u)\Qa(x,y) -\frac{\alpha \bx (y-2u)}{1-2us\bx} \Qa\left(
  \frac{2su-\bx}{1-2su\bx}, y\right) \\
+ \frac{\alpha \beta \bx\by }{(1-2su\bx)(1-2u\by)} \Qa\left(
  \frac{2su-\bx}{1-2su\bx}, \frac{2u-\by}{1-2u\by}\right)
- \frac{ \beta\by(x-2su)}{1-2u\by} \Qa\left( x, \frac{2u-\by}{1-2u\by}\right) \\
= \frac{(4su-x-\bx)(4u-y-\by)}{K(x,y)},
\end{multline*}
where $K(x,y)$ is given by~\eqref{Km1}. All the series occurring in
this equation are power series in $u$ with coefficients in $\qs[s, x,
\bx, y, \by]$. The series $(x-2su)(y-2u)\Qa(x,y)$ consists of
monomials in which the exponents of $x$ and $y$ are always
non-negative. In the three other series occurring in the left-hand side,
either the exponent of $x$, or the exponent of $y$ (or both) is
negative. This tells us that $(x-2su)(y-2u)\Qa(x,y)$ is the
non-negative part in $x$ and $y$ of the rational series occurring in the
right-hand side. In particular, extracting from the above equation the
coefficient of $x^0y^0$ gives
\begin{align*}
4su^2 \Qa(0,0) &\equiv 4su^2 \Qa(-1,s,u;0,0) \\
&= [x^0y^0]\, 
\frac{(4su-x-\bx)(4u-y-\by)}{1-u(sx+s\bx+y+\by)+2u^2s(x\by+y\bx)}.
\end{align*}
Let us now perform this coefficient extraction, beginning with the
constant term in $y$: 
\begin{eqnarray*}
  4su^2 \Qa(0,0)&=& [x^0y^0]  \frac{(4su-x-\bx)(4u-y-\by)}{\big(1-us(x+\bx)\big)
\left(1- \frac{uy(1-2us\bx)}{1-us(x+\bx)} -
  \frac{u\by(1-2usx)}{1-us(x+\bx)}\right)}
\\
&=& [x^0]\frac{4su-x-\bx}{1-us(x+\bx)}\times
\\
& & \hskip -19mm \sum_{i,j\ge 0}{i+j\choose
  i}\frac
  {u^{i+j}(1-2us\bx)^j(1-2usx)^i}{(1-us(x+\bx))^{i+j}}[y^0]\left( 4uy^{j-i}-y^{j-i+1}-y^{j-i-1}\right).
\end{eqnarray*}
We thus need to extract from the double sum over $(i,j)$ the summands obtained for
$i=j$, for $i=j+1$ and for 
$j=i+1$. This yields three simple sums. Upon exchanging $i$ and $j$ in the third one, this gives
\begin{eqnarray}
  4su^2 \Qa(0,0)&=& [x^0]\frac{4su-x-\bx}{1-us(x+\bx)} \sum_{j\ge 0}
  \frac{u^{2j}(1-2us\bx)^j(1-2usx)^j}{(1-us(x+\bx))^{2j}} \nonumber \\
&& \hskip -16mm \times\left( 4u{2j\choose j} -u{2j+1\choose
    j}\frac{1-2usx}{1-us(x+\bx)}- u{2j+1\choose 
  j}\frac{1-2us\bx}{1-us(x+\bx)}\right) \nonumber
\\
&=& 2[x^0] (4su-x-\bx)\sum_{j\ge 0}C_j u^{2j+1} \,
  \frac{(1-2us\bx)^j(1-2usx)^j}{(1-us(x+\bx))^{2j+1}} 
\label{extr}
\end{eqnarray}
% mbm "th" in exponent below
with $C_j$ the $j^{\mbox{\scriptsize th}}$ Catalan number.
It remains to extract the constant term in $x$. Let us return for a
while to the
expression~\eqref{expr-Qm1} of $\Q(-1,s,u)=\Qa(-1,\sqrt s, \sqrt
u;0,0)$ that we want to establish. It is equivalent to
\begin{multline}
\label{toprove}
4su^2 \Qa(-1, s,u;0,0)\equiv  4su^2 \Qa(0,0) \\  = 4 \sum_{j\ge 0} C_j u^{2j+1} \sum_{i\ge 0}{i+j\choose
   i} C_i (us)^{2i+1}  .
\end{multline}
Comparing with~\eqref{extr} shows that  what remains to prove is that
for  $j\ge 0$,
 $$
[x^0]  (4su-x-\bx) \frac{(1-2su\bx)^j(1-2sux)^j}{(1-su(x+\bx))^{2j+1}} 
= 2 \sum_{i\ge 0} {i+j\choose i} C_i (us)^{2i+1},
$$
or equivalently, by taking the generating function of this collection of identities:
\begin{multline*}
[x^0]  (4su-x-\bx)\sum_{j\ge 0} u^{2j} \frac{(1-2su\bx)^j(1-2sux)^j}{(1-su(x+\bx))^{2j+1}} 
\\ = 2 \sum_{j\ge 0} u^{2j}\sum_{i\ge 0}  {i+j\choose i} C_i(us)^{2i+1}.
\end{multline*}
This is of course equivalent to prove that~\eqref{extr} and~\eqref{toprove} coincide, but
the absence of the factor $C_j$ makes this new task easier. In
particular, we are now
handling algebraic series. Indeed, all the sums occurring in the above
identities can be evaluated in closed form, and what we now need to prove
is the following lemma.
\begin{Lemma}\label{lem:CT}
Let $R(s,u;x)$ be the following rational function:
\beq\label{R-def}
R(s,u;x)= \frac{(4su-x-\bx)(1-su(x+\bx))}{(1-su(x+\bx))^2-u^2\,
 (1-2su\bx)(1-2sux)}.
 \eeq
Then its constant term in $x$ is
\beq\label{extr3}
A(s,u):=[x^0]  R(s,u;x)=  \frac 1{us} \left(1-\sqrt{1- \frac{4u ^2s^2}{1-u^2}}\right).
\eeq
\end{Lemma}
\begin{proof}
  There are  several ways of performing this extraction effectively.
As in~\cite[Thm.~6.3.3]{stanley-vol2}, we  use a partial fraction extraction in $x$.

 The denominator of $R$ is a Laurent polynomial in $x$,
symmetric in $x$ and $\bx$, of degree 2. It has four roots, which are
Laurent series in $u$ with coefficients in $\cs[s]$. Two of them are
actually power series in $u$, and vanish at $u=0$:
$$
X_{1,2}
=
us\pm isu^2+s^3u^3\pm is(s^2+1/2)u^4+s^3(2s^2+1)u^5+O(u^6)
$$
where $i^2=-1$. The other two are $\bX_1:=1/X_1$ and $\bX_2:=1/X_2$.
%, and they read $1/(us)+O(1)$. 
We will now perform a partial fraction expansion of $R$
with respect to $x$, after writing  $R$ as
\beq\label{R-mixte}
R(s,u;x)= \frac{X_1X_2(4su-x-\bx)(1-su(x+\bx))}{u^2s^2(1-xX_1)(1-xX_2)(1-\bx X_1)(1-\bx X_2)}.
\eeq
In fact, we can also write the factor $su$  in terms of $X_1$ and
$X_2$, and this will simplify the result of the partial fraction
expansion a bit.  Indeed, since $X_1$ and $X_2$ cancel the denominator
of $R$, we derive from~\eqref{R-def} that
\beq\label{denom}
u^2 = \frac {(1-su(X_1+\bX_1))^2}{(1-2su\bX_1)(1-2suX_1)  }= \frac
{(1-su(X_2+\bX_2))^2}{(1-2su\bX_2)(1-2suX_2)  }.
\eeq
By solving the second equation for $su$, we find
$$
su= \frac{X_1+X_2}{2(1+X_1X_2)}.
$$
(There is another solution, $su= \frac{1+X_1X_2}{2(X_1+X_2)}$, but it is excluded
since the $X_i$'s are multiples of $u$.) Returning to~\eqref{R-mixte},
this gives
\begin{multline*}
  R(s,u;x) = \\ \frac{2X_1X_2\left(2(X_1+X_2)-(x+\bx)(1+X_1X_2)\right)\left(2(1+X_1X_2)-(x+\bx)(X_1+X_2)\right)}{(X_1+X_2)^2(1-xX_1)(1-xX_2)(1-\bx
    X_1)(1-\bx X_2)}
\\
=
 \frac{2(1+X_1X_2)}{X_1+X_2}
%\frac 1{us} 
+ 
\frac{\alpha_1}{1-xX_1} + \frac{\alpha_2}{1-xX_2} +
\frac{\alpha_1\bx X_1}{1-\bx X_1}  +
\frac{\alpha_2\bx X_2}{1-\bx X_2} 
\end{multline*}
where
$$
\alpha_1=
%\frac{ (4su-X_1-\bX_1)(1-su(X_1+\bX_1))}{s^2u^2
%(X_1-\bX_1)(1-\bX_1X_2)(X_1-\bX_2)} 
- \frac{2X_2(1-X_1^2)}{(X_1+X_2)^2}
$$
and symmetrically for $\alpha_2$. Since $X_{1}$ and $X_2$ are
multiples of $u$, we can read off the coefficient of $x^0$ in $R(s,u;x)$:
$$% \beq\label{A-expr}
A(s,u)=[x^0]R(s,u;x)=  \frac{2(1+X_1X_2)}{X_1+X_2}
%\frac 1{us} 
+ {\alpha_1} + {\alpha_2}=\frac {4X_1X_2}{X_1+X_2}.
$$%\eeq
We finally eliminate $X_1$ and $X_2$ using the
identities~\eqref{denom}, and this gives an algebraic equation
satisfied by $A(s,u)$. This equation has four distinct factors. One is
quartic in $A$, and the other three are quadratic. Only one factor has
a solution that is a power series in $u$ with coefficients in
$\qs[s]$, and this solution is precisely~\eqref{extr3}.
\end{proof}
This concludes the proof of Proposition~\ref{prop:Qm1}.

%%%%%%%%%%%%%%%%%%%%%%%%%%%%%%%%%%%%%%%%%%%%%%%%%%%%%%%%%%%%%%%%%%%ùù
\section{Proof of Lemma~\ref{lem:C<1/2}}
\label{APP:CSmall}
%%%%%%%%%%%%%%%%%%%%%%%%%%%%%%%%%%%%%%%%%%%%%%%%%%%%%%%%%%%%%%%%%%%%%%%%%%%%%

Recall that $C(b,v)$ is the generating function for connected arch
systems beginning with a red arch, counted by the number of arches
(variable $v$) and the number of
left-right pairs (variable~$b$). We denote by $\rho_ C(b)$ the radius of
convergence of $C(b,\cdot)$. For convenience, we repeat here the
lemma we want to prove.

\smallskip \noindent
{\bf {Lemma}~\ref{lem:C<1/2}.}
\emph{Let $b >0$. Then $\rho_C(b) \le 1/4$ and for
   $v \in  [0, \rho_C(b))$, 
$$
v  < C(b,v) <  \frac{1}{2}.
$$
The series $A(b,\cdot)$ and $U(b,\cdot)$ defined
  by~\eqref{UBAT}  have radius of  convergence at least $\rho_C(b)$.\\
 The series $C(0, \cdot)$, $A(0, \cdot)$ and $U(0, \cdot)$ have
 respectively radius $+\infty$, $+\infty$ and $1/2$.}

\begin{proof}
  Let us say that  a quarter plane
  loop is \emm self-avoiding, if it only visits the point $(0,0)$
  at the beginning and at the end, and does not visit any other point
  twice. It follows from the proof of~\eqref{eq:Qa} that, if a quarter plane
  loop is not connected, it admits a proper factor that is itself
  a loop. In particular, it is not self-avoiding. Consequently,
  every quarter plane self-avoiding loop is connected. It is standard if it
  begins with an \EE\ step.

\begin{figure}[ht]
\begin{center}
{\includegraphics[height=3cm]{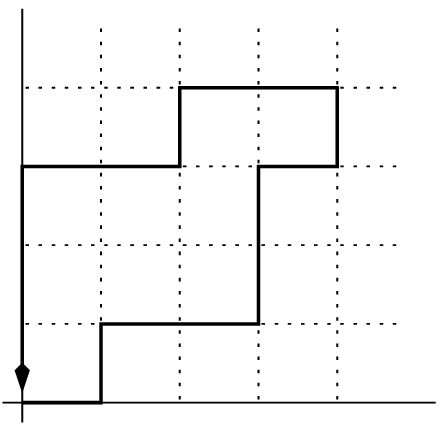}}
\caption{A staircase polygon.} 
\label{fig:stair}
\end{center}
\end{figure}

Let us use this to  bound  the radius
$\rho_C(b)$ from above. A quarter plane self-avoiding loop is 
a \emm staircase polygon, if it consists of a sequence of \EE\ and \NN\
steps, followed by a sequence of \WW\ and \SS\ steps
(Figure~\ref{fig:stair}). It is 
well known that the  generating function of staircase polygons (according to the
half-length) is~\cite[Exercise~6.19.l]{stanley-vol2}:
$$
SP(v)=\sum_{n\ge 1} \frac 1 {n+1}{2n\choose n} v^{n+1}=
\frac{1-\sqrt{1-4v}}{2},
$$
which has radius of convergence $1/4$. Since non-degenerate staircase polygons have
exactly one \NW\ corner, and no \ES\  corner, the above discussion implies
that the series  $C(b,v)$ dominates $b (SP(v) - v)$ term by term. Hence 
$\rho_C(b)$  is at most $1/4$.

Let us now prove the inequalities on $C(b,v)$. The lower bound is
obtained by counting only the arch system reduced to a single
arch. The upper bound follows from
another inequality, which is 
combinatorial in the sense  it holds coefficient by
coefficient. Namely, we will prove that the series
\beq\label{C-ineq}
%(1+v)C(b,v)-2C(b,v)^2+v(v-1)
C-v-bv^2-2C(C-v)
\eeq
has non-negative coefficients. For $v
\in (0, \rho_C(b))$, dividing by $C-v$ gives
$$
1-2C\ge \frac{bv^2}{C-v}>0.
$$
In particular $C(b,v)<1/2$.

So let us  prove that the series~\eqref{C-ineq} has non-negative
coefficients. 
We begin with some terminology. Two arches in an arch system are
\emph{parallel} if they are adjacent 
at both  ends,  nested, and have the same colour.  For instance, the arches 3
 and 4 in Figure~\ref{fig:ex} are parallel. We define the \emm
 negative, $-x$ of an arch system $x$ to be the one 
obtained by interchanging colours --- that is, reflecting $x$
in a horizontal line. Let $v$ denote the arch system
consisting of a single standard (that is, red) arch (conveniently, the
generating function for this arch is also $v$). 

Let $\C$ denote the collection of all standard connected arch
systems, counted by $C(b,v)$.  We now define two injective maps: 
\[
\Phi: (\C - v) \times \C \to \C \qquad \hbox{and} \qquad 
\Psi: \C \times (\C - v) \to \C
\]
whose images are disjoint, and which do not change the total number of
arches nor the total number of left-right pairs.

% \medskip
\textbf{Construction of $\Phi$}.  
 Take $x \in \C - v$ and $y \in \C$. If the \emm last, arch
of $x$ (that is, the one with the rightmost right end) is red, let $x' = -x$, otherwise let $x' = x$. Now form
$\Phi(x,y)$ as follows (Figure~\ref{fig:Phi}). Place $x'$ to the left
of $y$.  Unhook the right end of the
last arch of $x'$ and pass it beneath $y$ before reconnecting with the
line.  Unhook the left end of the first arch of $y$ and pass it above
$x'$ before reconnecting with the line. 

Let us prove that the resulting  arch system $\Phi(x,y)$ is connected. Its graph
is obtained by juxtaposing the graphs of $x'$ and $y$ (which are both
connected) and adding an edge between them (corresponding to the
crossing between the first and last arches of $\Phi(x,y)$). This graph
is connected, and so is $\Phi(x,y)$. This graph can also be used to
prove that $\Phi$ is injective:  if we delete the
edge connecting the first and last arch, we obtain two connected
components. The left one is $x'$ (from which we can find $x$)
and the right one is $y$. 

Note that since $x \neq v$, the first two arches
of $\Phi(x,y)$ do not cross. One also checks that the number of
left-right pairs behaves additively (this also uses the fact that $x \ne v$).

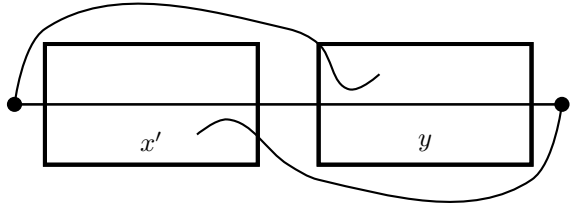
\begin{figure}[ht]
\centerline{
\psset{xunit=0.4}
\psset{yunit=0.4}
\begin{pspicture}(0,-2)(20,6)
\pspolygon[linewidth=1.5pt](1,0)(8,0)(8,4)(1,4)
\pspolygon[linewidth=1.5pt](10,0)(17,0)(17,4)(10,4)
\psline(0,2)(18,2)
\pspolygon[linewidth=1.5pt](1,0)(8,0)(8,4)(1,4)
\pspolygon[linewidth=1.5pt](10,0)(17,0)(17,4)(10,4)
\psline(0,2)(18,2)
\uput[90](4.5,0){$x'$}
\uput[90](13.5,0){$y$}
\pscurve(12,3)(11,2.5)(10,4)(9,4.5)(1,4.5)(0,2)
\pscurve(6,1)(7,1.5)(9,0)(10,-0.5)(17,-0.5)(18,2)
\pscircle*(0,2){0.1}
\pscircle*(18,2){0.1}
\end{pspicture}
}
\caption{The construction of $\Phi(x,y)$.}
\label{fig:Phi}
\end{figure}

% \medskip
\textbf{Construction of $\Psi$}. 
Let $(x, y) \in  \C \times (\C - v)$. Place $-x$ on the line,
and place a copy of $y$ between its first two points. Now unhook the
left end  the first arch of $y$ and pass it above the first point of
$-x$ before reconnecting with the line (Figure~\ref{fig:Psi}).

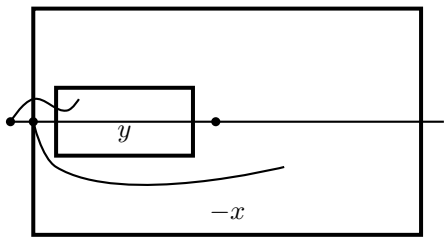
\begin{figure}
\centerline{
\psset{xunit=0.3}
\psset{yunit=0.3}
\begin{pspicture}(0,0)(20,11)
\pspolygon[linewidth=1.5pt](1,0)(18,0)(18,10)(1,10)
\pspolygon[linewidth=1.5pt](2,3.5)(8,3.5)(8,6.5)(2,6.5)
\psline(0,5)(19,5)
\pscurve(3,6)(2.5,5.5)(1,6)(0,5)
\pscurve(1,5)(2,3)(12,3)
\pscircle*(9,5){0.06}
\pscircle*(1,5){0.06}
\pscircle*(0,5){0.06}
\uput[90](5,3.5){$y$}
\uput[90](9.5,0){$-x$}
\end{pspicture}
}
\caption{The construction of $\Psi(x,y)$.}
\label{fig:Psi}
\end{figure}

As above, the graph of $\Psi(x,y)$ is obtained by adding an edge
joining a vertex of the graph of $x$ to a vertex of the graph of
$y$. This graph is connected, and so is  $\Psi(x,y)$. If we delete
from the graph of  $\Psi(x,y)$
  the edge between the first two arches we obtain two connected
components from which we can recover $x$ and $y$. Hence $\Psi$ is
injective.

 The number of % arches and
 left-right pairs still
behaves additively (this uses $y \ne v$). 
Since the first two arches of  $\Psi(x,y) $ \emm do,
cross,  the range of $\Psi$ is disjoint 
from that of $\Phi$.  In particular, the union of their images is
enumerated by $2C(C-v)$. This series
counts a subset of $\cC$ which consists of arch systems with at least
three arches.
We thus conclude that
$
C-v-bv^2 - 2C(C-v)
$
has non-negative coefficients, as claimed.

% \medskip
To finish, it is clear that $A(b,v)=1+(b-1)(1+2C(b,v))$ has at least
radius $\rho_C$. Moreover, since $C(b,v)$ has non-negative
coefficients, then the first part of the lemma implies that
$|C(b,v)|<1/2$ for $b>0$ and $|v|< \rho_C(b)$. The function 
$$
U(b,v)=\frac v{(1+2C(b,v))^2}%=v \sum_{i\ge 1} i (-2)^i C(b,v)^i
$$
is thus analytic in this disk, and thus has radius of convergence at
least $\rho_C(b)$. The results stated for $b=0$ are obvious, since $C(0,v)=v$.
\end{proof}

%%%%%%%%%%%%%%%%%%%%%%%%%%%%%%%%%%%%%%%%%%%%%%%%%%%%%%%
\section*{References}
\bibliography{combinatorics}

\end{document}